\documentclass[a4paper,11pt]{amsart}
 \usepackage{amssymb,amsfonts,amsmath}
\def\d{\delta}
\def\H{\mathcal{H}}

\def\C{\mathbb{C}}
\def\c2{\mathbb{C}^2}
\def\R{\mathbb{R}}

\def\Z{\mathbb{Z}}
\def\N{\mathbb{N}}
\def\D{\mathbb{D}}
\def\P{\mathbb{P}}

\def\E{\mathcal{E}}
\def\W{\mathcal{W}}

\def\1{\bold{1}}
\def\B{\mathbb{B}}

\def\a{\alpha}

\def\e{\varepsilon}
\def\l{\lambda}
\def\f{\varphi}
\def\g{\gamma}

\def\p{\psi}

\def\om{\omega}

\def\z{\overline{z}}

\def\vol{\rm{vol}}

\newtheorem{lem}{Lemma}[section]
\newtheorem{prop}[lem]{Proposition}
\newtheorem{defi}[lem]{Definition}
\newtheorem{def/not}[lem]{Definition/Notations}

\newtheorem{thm}[lem]{Theorem}

\newtheorem{cor}[lem]{Corollary}
\newtheorem{rem}[lem]{Remark}

\newtheorem{exa}[lem]{Example}

\newtheorem*{ackn}{Acknowledgement}

\newcommand{\Real}{\mathbb{R}}

\newcommand{\Circle}{\mathbb{S}^{1}}

\newcommand{\HH}{\mathcal{H}}

\newcommand{\Cinf}{C^{\infty}}

\newcommand{\norm}[1]{\left\Vert#1\right\Vert}

\newcommand{\set}[1]{\left\{#1\right\}}

\begin{document}

\title[The  Riemannian space of K\"{a}hler metrics]{The metric completion of the Riemannian space of K\"{a}hler metrics}

\author{Vincent GUEDJ*}

 \date{\today \\ *The author is partially supported by the ANR project MACK}
 
 \address{Institut Universitaire de France \& Institut de Math\'ematiques de Toulouse, 31062 Toulouse cedex 09\\
France}
\email{vincent.guedj@math.univ-toulouse.fr}

\begin{abstract}
Let $X$ be a compact K\"ahler manifold and
$\a \in H^{1,1}(X,\R)$ a K\"ahler class.
We study the metric completion of the space 
$\HH_\a$ of K\"ahler metrics in $\a$, when endowed with the
Mabuchi $L^2$-metric $d$. 

Using recent ideas of Darvas, we  show that the metric completion $(\overline{\HH}_\a,d)$ of $(\HH_\a,d)$ 
is a CAT(0) space which can be identified
with  $\E^2(\a)$, a subset of the class  $\E^1(\a)$ of positive closed currents with finite energy.

We further prove, in the toric setting, that $\overline{\HH}_{\a,tor}=\E_{tor}^2(\a)$.
\end{abstract}

\maketitle

\tableofcontents
\clearpage


\section*{Introduction}
\label{sec:intro}
 
 Let $X$ be a compact K\"ahler manifold and $\a \in H^{1,1}(X,\R)$ a K\"ahler class.
The space $\HH_\a$ of K\"ahler metrics $\omega$ in $\a$ can be seen as an infinite dimensional riemannian manifold
whose tangent spaces $T_\omega \H_\a$ can all be identified with ${\mathcal C}^{\infty}(X,\R)$.
Mabuchi has introduced in \cite{Mab87} an $L^2$-metric on $\H_\a$, by setting
$$
\langle f, g \rangle_\omega:=\int_X f \, g \, \frac{\omega^n}{V_\a},
$$
where $n=\dim_\C X$ and $V_\a=\int_X \omega^n=\a^n$ denotes the volume of $\a$.

Mabuchi  studied the corresponding geometry of $\H_\a$, showing in particular that it
can formally be seen as a locally symmetric space of non positive curvature. Semmes \cite{Sem92} re-interpreted the geodesic equation as a complex homogeneous equation, while Donaldson \cite{Don99}
strongly motivated the search for smooth geodesics through its connection with the uniqueness of constant scalar curvature K\"ahler metrics.

\smallskip

In a series of remarkable works \cite{Chen00,CC02,CT08,Chen09,CS09} X.X.Chen
and his collaborators have studied the metric and geometric properties of the space $\H_\a$, showing in particular that it is a path metric space 
(a non trivial assertion in this infinite dimensional setting) of nonpositive curvature in the sense of Alexandrov. A key step from \cite{Chen00} has been to produce almost ${\mathcal C}^{1,1}$-geodesics
which turn out to minimize the intrinsic  distance $d$. 

It is a very natural problem to try and understand the metric completion of the metric space $(\H_\a,d)$
(see \cite{Cla13} for motivations and similar considerations on the space of all riemannian metrics).
This is the main goal of this expositary article.

\medskip

\noindent {\bf Main theorem.}
{\it
The metric completion $(\overline{\H}_\a,d)$ of $(\H_\a,d)$ is a Hadamard space which consists of finite energy currents, more precisely
$$
(\overline{\H}_\a, d)  \text{ is isometric to } \E^2(\a).
$$

Moreover weak (finite energy) geodesics are length minimizing.
}

\medskip

I had conjectured this result  in a preliminary draft of these lecture notes. It has been since then 
proved by Darvas in \cite{Dar14}. We have  thus expanded our notes so as to
include some of Darvas arguments.

\smallskip

Recall that a Hadamard space is a complete CAT(0) space, i.e. a complete geodesic space which has non positive curvature in the sense of Alexandrov (see \cite[Chapter II.1]{BH99}).

Finite energy classes of currents $\E^p(\a)$ were introduced in \cite{GZ07}, inspired by similar concepts developed by Cegrell in
domains of $\C^n$ \cite{Ceg98}. They have become an important tool in studying singular K\"ahler-Einstein metrics \cite{EGZ09,BEGZ10}.
We refer the reader to section \ref{sec:energy} for a precise definition.
We just stress here that when the complex dimension is $n=1$, the class 
$\E^1(\a)$ consists of non negative Radon measures whose potentials have square integrable gradients;
we show that the Mabuchi metric dominates the Sobolev norm in this case.

Weak geodesics are generalized solutions of the homogeneous complex Monge-Amp\`ere equation, when
the boundary data are no longer smooth points in $\H_\a$, but rather finite energy currents in $\overline{\H}_\a$. They have played  an important role in 
Berndtsson's recent generalization of Bando-Mabuchi's uniqueness result \cite{Bern13}.
They also play a crucial role in the variational approach to solving degenerate complex Monge-Amp\`ere equations in finite energy classes and its application to the K\"ahler-Einstein problem on singular varieties \cite{BBGZ,BBEGZ}.

\smallskip

 When $X_P$ is a toric manifold associated to a Delzant polytope $P \subset \R^n$, the Legendre transformation
allows to linearize the Monge-Amp\`ere operator.  We explain in section \ref{sec:toric}
that all notions (K\"ahler metrics, finite energy currents, geodesics, etc) have an elementary interpretation
on $P$, following previous observations by Guillemin, Guan, Donaldson, Berman-Berndtsson \cite{Gui94,Guan99,Don02,CDG03,BerBer13} (to name a few). Setting $\H_{tor}$ for the set of toric K\"ahler metrics in $\a$ and 
$\E^2_{tor}$ for the corresponding set of finite energy toric currents, we then prove:

\medskip
\noindent {\bf Toric theorem.}
{\it
  The metric completion of $(\H_{tor},d)$ is isometric to $(\E^2_{tor},d)$.
}
 \medskip

More precisely toric K\"ahler metrics correspond to certain "smooth" convex functions on the polytope $P$,
geodesics in $\H_{tor}$ correspond to straight lines on $P$,
 and the Mabuchi distance corresponds to the Lebesgue $L^2$-norm on the latter, while
finite energy toric currents in $\E^2$ correspond precisely to convex functions on $P$ that
are $L^2$ with respect to Lebesgue measure.
Even though the toric setting is very particular, it is a good test which naturally lead us to conjecture 
the corresponding non toric result.

\smallskip
 
There are several related questions that we do not address here. It would be for instance interesting to extend the quantization results of Phong-Sturm \cite{PS06}, Berndtsson \cite{Bern09}, Chen-Sun \cite{CS09} 
and Song-Zelditch \cite{SZ12}, Rubinstein-Zelditch \cite{RZ12} to $\overline{\H}_\a$. The reader will find further interesting problems by consulting the surveys \cite{G12,PSS12}.

\smallskip

The organization of the paper is as follows. {\it Section \ref{sec:kahler_metrics}}
is a recap on the Mabuchi metric and those results of Chen that we shall be using.
We give only few indications on the proofs and 
 refer the interested reader to \cite{G12} for a much broader overview of this exciting field.

We introduce in {\it Section \ref{sec:energy}} classes of finite energy currents and compare
their natural topologies with the one induced by the Mabuchi distance. 

We study weak geodesics in {\it Section \ref{sec:geod}}.
We start by explaining an important observation of Darvas that the Mabuchi distance decreases under the "min operation" and then prove the main theorem.
 
We finally provide a detailed analysis of the toric setting in {\it Section \ref{sec:toric}}.
 
\begin{ackn} 
These are expanded sets of  notes of a series of lectures given at KIAS in April 2013. The author thanks 
J.-M.Hwang and M.Paun for their invitation and the staff of KIAS for providing excellent conditions of work.

We thank T.Darvas  for communicating his interesting preprint \cite{Dar14}
and for useful comments on a preliminary draft of our notes.  

The metric space $\overline{\H}_\a$ has been considered by J.Streets in his  
study of the Calabi flow. We thank him for informing us about his work 
\cite{Str12,Str13}.
\end{ackn}

\section{The space of K\"{a}hler metrics}
\label{sec:kahler_metrics}

Let $(X,\omega)$ be a {compact} K\"ahler manifold of dimension $n$. 
It follows from the $\partial\overline{\partial}$-lemma that any
 other K\"{a}hler metric on $M$ {in the same cohomology class} as $\omega$ can be written as
\begin{equation*}
    \omega_{\f} = \omega + dd^c \f ,
\end{equation*}
where $d=\partial+\overline{\partial}$ and $d^c=\frac{1}{2i\pi} (\partial -\overline{\partial})$.
Let $\HH$ be the space of \emph{K\"{a}hler potentials}
\begin{equation*}
    \HH = \set{\f \in \Cinf(X) \: ; \omega_{\f} = \omega + dd^c  \f >0}.
\end{equation*}
This is a convex open subset of the Fr\'{e}chet vector space $\Cinf(X)$, thus itself a Fr\'{e}chet manifold, which is moreover parallelizable : $T\HH = \HH \times \Cinf(X)$. Each tangent space is identified with $\Cinf(X)$.

As two K\"{a}hler potentials define the same metric when (and only when) they differ by an additive constant,
we set
\begin{equation*}
    \HH_{\a} = \HH /\Real
\end{equation*}
where $\Real$ acts on $\HH$ by addition. The set $\HH_{\a}$ is therefore the {space of K\"{a}hler metrics on $X$ in the cohomology class $\a:=\{\omega\} \in H^{1,1}(X,\R)$}.

We briefly review in this section known facts about the riemannian structure of this space, as introduced by Mabuchi in \cite{Mab87} and further investigated by Semmes \cite{Sem92}, Donaldson \cite{Don99},
Chen \cite{Chen00,CC02}.  Introductory references for this material  are the lecture notes by Kolev \cite{Kol12}
and Boucksom \cite{Bou12}.

\subsection{The Riemannian structure}

\begin{defi} \cite{Mab87}
The \emph{Mabuchi metric}  is the $L^{2}$ Riemannian metric on $\HH$. It is defined by
\begin{equation*}
    <\psi_{1},\psi_{2}>_{\f} = \int_{X} \psi_{1}\psi_{2}\, \frac{(\omega+dd^c \f)^n}{V_\a}
\end{equation*}
where $\f \in \HH$, $\psi_{1},\psi_{2} \in \Cinf(X)$ 
and ${(\omega+dd^c \f)^n}/{V_\a}$ is the volume element, normalized so that
it is a probability measure. Here
$$
V_\a:=\a^n=\int_X \omega^n.
$$
\end{defi}

In the sequel we shall also use the notation $\omega_\f:=\omega+dd^c \f$ and 
$$
MA(\f):=\frac{\omega_\f^n}{V_\a}
$$

Geodesics between two points $\f_{0}$, $\f_{1}$ in $\HH$ are defined as the extremals of the
{Energy functional}
\begin{equation*}
 \f \mapsto  H(\f)=\frac{1}{2} \int_{0}^{1}\int_{X}(\dot{\f_{t}})^{2}\, MA({\f_{t}}) \, dt.
\end{equation*}
where $\f = \f_{t}$ is a path in $\HH$ joining $\f_{0}$ and $\f_{1}$. The geodesic equation is obtained by computing the Euler-Lagrange equation for this Energy functional (with fixed end points):

\begin{lem}
The geodesic equation is
\begin{equation}\label{equ:geodesic_equation}
    \ddot{\f} = \norm{\nabla \dot{\f}}_{\f}^{2}
\end{equation}
where the gradient is relative to the metric $\omega_{\f}$. This identity also writes
$$
\ddot{\f} \, MA(\f)=\frac{n}{V_\a} d \dot{\f} \wedge d^c \dot{\f}  \wedge \omega_\f^{n-1}.
$$
\end{lem}

\begin{proof}
 We need to compute the Euler-Lagrange equation of the Energy functional $H$.
Let $(\phi_{s,t})$ be a variation of $(\f_t)$ with fixed end points,
$$
\phi_{0,t} \equiv \f_t
\; \text{ and } \;
\phi_{s,0} \equiv \f_0, \;
\phi_{s,1} \equiv \f_1.
$$
Set $\p_t:=\frac{\partial \phi}{\partial s}_{|s=0}$ and observe that $\p_0 \equiv \p_1 \equiv 0$. Thus
$$
\phi_{s,t}=\f_t+s\p_t+o(s)
\; \text{ and } \; 
\frac{\partial \phi_{s,t}}{\partial t}=\dot{\f}_t+s\dot{\p}_t+o(s).
$$

A direct computation yields
\begin{eqnarray*}
\lefteqn{
H(\phi_{s,t}) =\frac{1}{2} \int_{0}^{1}\int_{X}(\dot{\phi_{s,t}})^{2}\, MA({\phi_{s,t}}) \, dt} \\
&=& H(\f_t)+s \int_0^1 \int_X \dot{\f}_t \dot{\p}_t MA(\f_t) dt
+\frac{ns}{2V} \int_0^1 \int_X (\dot{\f}_t)^2 dd^c {\p}_t \wedge \omega^{n-1}_{\f_t} dt,
\end{eqnarray*}
noting that 
$$
MA({\phi_{s,t}}) =MA(\f_t)+\frac{ns}{V} dd^c \p_t \wedge \omega^{n-1}_{\f_t}.
$$
Integrating by parts yields, since $\p_0 \equiv \p_1 \equiv 0$,
$$
\int_0^1 \int_X \dot{\f}_t \dot{\p}_t MA(\f_t) dt=-\int_0^1 \int_X {\p}_t 
\left\{ \ddot{\f}_t  MA(\f_t) +\frac{n}{V} \dot{\f}_t dd^c \dot{\f}_t \wedge \omega_{\f_t}^{n-1} \right\} dt
$$
while
$$
\int_0^1 \int_X (\dot{\f}_t)^2 dd^c {\p}_t \wedge \omega^{n-1}_{\f_t} dt=
2\int_0^1 \int_X \p_t \left\{ d \dot{\f}_t \wedge d^c \dot{\f}_t +\dot{\f}_t \wedge dd^c \dot{\f}_t \right\}
\wedge \omega_{\f_t}^{n-1} dt,
$$
hence
$$
H(\phi_{s,t}) =H(\f_t)+s \int_0^1 \int_X \p_t \left\{ -\ddot{\f}_t  MA(\f_t) 
+\frac{n d \dot{\f}_t \wedge d^c \dot{\f}_t \wedge \omega_{\f_t}^{n-1}}{V} \right\} dt+o(s).
$$
Therefore $(\f_t)$ is a critical point of $H$ if and only if
$$
\ddot{\f} \, MA(\f)=\frac{n}{V_\a} d \dot{\f} \wedge d^c \dot{\f}  \wedge \omega_\f^{n-1}.
$$
\end{proof}

As for Riemannian manifolds of finite dimension, one can find the local expression of the Levi-Civita
connection  by  {polarizing}  the geodesic equation.  We define the covariant derivative of the vector field $\psi_{t}$ along the path $\f_{t}$ in $\HH$ by the formula
\begin{equation*}
    \frac{D\psi}{Dt} = \partial_{t}\psi -  <\nabla \psi, \nabla \dot{\f}>_{\f} .
\end{equation*}
This covariant derivative is symmetric by its very definition, i.e.
\begin{equation*}
    \frac{D}{Ds}\partial_{t}\f = \frac{D}{Dt}\partial_{s}\f,
\end{equation*}
for every family $\f = \f_{s,t}$ and it can be checked directly that it preserves the metric,
\begin{equation*}
    \frac{d}{dt}<\psi_{1},\psi_{2}>_{\f} = <\frac{D\psi_{1}}{Dt},\psi_{2})>_{\f}  + <\psi_{1}, \frac{D\psi_{2}}{Dt}>_{\f}.
\end{equation*}

\subsection{Dirichlet problem for the complex Monge-Amp\`{e}re equation}

We are interested in the  {boundary value problem} for the geodesic equation: 
given $\f_{0},\f_{1}$ two distinct points in $\HH$, 
can one find a path $(\f(t))_{0 \leq t \leq 1}$  in $\HH$ which is a solution of  ~\eqref{equ:geodesic_equation} with end points $\f(0) = \f_{0}$ and $\f(1) = \f_{1}$ ?

It has been observed by Semmes \cite{Sem92} that this can be re-formulated as a {homogeneous complex Monge-Amp\`{e}re equation}. 

\smallskip

For each path $(\f_{t})_{t \in [0,1]}$ in $\HH$, we set
\begin{equation*}
    \f(x,t,s) = \f_{t}(x), \qquad x \in X, \quad e^{t+is} \in A = [1,e] \times \Circle;
\end{equation*}
i.e. we associate to each path $(\f_{t})$ a 
function  $\f$ on the complex manifold $X\times A$, which is radial in the annulus coordinate:
we consider the annulus $A$ as a Riemann surface with boundary 
and use the complex coordinate $z = e^{t+is}$ to parametrize the annulus $A$. Set $\omega(x,z):=\omega(x)$.

\begin{prop}\cite{Sem92} \label{prop:Semmes}
The path $\f_{t}$ is a geodesic in $\HH$ if and only if the associated radial function $\f$ on $X\times A$ is a solution of the homogenous complex Monge-Amp\`{e}re equation 
$$
(\omega+dd^c_{x,z} \f)^{n+1} = 0.
$$
\end{prop}

\noindent We have stressed the fact that we take here derivatives in
all variables $x,z$.

\begin{proof}
Write $d_{x,z}\f=d_x\f+d_z \f$ and observe that
$d_z \f=\dot{\f}_t (dz+d\overline{z})$ (by invariance), 
$(\omega+d_xd_x^c \f)^{n+1}=0$ and
$$
\omega+dd^c_{x,z} \f=\omega_{\f_t}+R,
$$
where 
$
R=d_x d^c_z \f+d_z d^c_x \f+d_zd^c_z \f
\text{ is such that }
R^3=0.
$
This yields
$$
(\omega+dd^c_{x,z} \f)^{n+1}=(n+1) \omega_{\f_t}^n \wedge R+
\frac{n(n+1)}{2} \omega_{\f_t}^{n-1} \wedge R^2.
$$
Now $d_zd^c_z \f=\frac{i}{\pi} \ddot{\f}_t dz \wedge d\z$ and
$$
(d_x d^c_z \f+d_z d^c_x \f)^2=-2 d_x \dot{\f}_t \wedge d_x^c \dot{\f}_t 
\wedge \frac{i}{\pi} \ddot{\f}_t dz \wedge d\z,
$$
thus
$$
(\omega+dd^c_{x,z} \f)^{n+1}=(n+1) \omega_{\f_t}^n \wedge \frac{i}{\pi} \ddot{\f}_t dz \wedge d\z
\left\{ \ddot{\f}_t-\frac{n d \dot{\f} \wedge d^c \dot{\f}  \wedge \omega_\f^{n-1}}{\omega_\f^{n}} 
\right\},
$$
and the result follows.
\end{proof}

It is well known that homogeneous complex Monge-Amp\`ere equations in bounded domains of 
$\C^n$ admit a unique solution which is at most ${\mathcal C}^{1,1}$-smooth.
Here is a striking example of Gamelin and Sibony:

\begin{exa}
Let $\B \Subset \C^2$ be the  open unit ball. Observe that
 $$
 \varphi_0(z,w) := (\vert z\vert^2 - 1 \slash 2)^2 = (\vert w\vert^2 - 1 \slash 2)^2
 $$
is real analytic on $\partial \B$. Set $ \psi (z) := \left(\max \{0, \vert z\vert^2 - 1 \slash 2\}\right)^2$ and
for $(z,w) \in \B$,
 $$
\f (z,w) = \max \{\psi (z), \psi (w)\}.
 $$

 The reader will check that $\f$  is ${\mathcal C}^{1,1}$-smooth but not ${\mathcal C}^2$-smooth,
$\f_{| \partial \B}=\f_0$, and $\f$ is the unique solution of the homogeneous 
complex Monge-Amp\`ere equation
$MA(\f)=0$ in $\B$, with $\f_0$ as boundary values.
\end{exa}

A fundamental result of Chen \cite{Chen00} shows the existence (and uniqueness) of
almost ${\mathcal C}^{1,1}$-solutions to the geodesic equation\footnote{These have bounded Laplacian, hence they are in particular ${\mathcal C}^{1,\a}$
for all $0<\a<1$.}. 

It was expected that Chen's regularity result was essentially optimal 
(although Chen and Tian have proposed in \cite{CT08} some improvements). This has been 
confirmed by Lempert, Vivas and Darvas in \cite{LV11,DL12}. 
We summarize this discussion in the following:

\begin{thm} [Chen, Lempert-Vivas]
Given $\f_0,\f_1 \in \H$ there exists $C>0$ and a continuous geodesic path $(\f_t)_{0 \leq t \leq 1}$
in $\overline{\H}$ joining $\f_0$ to $\f_1$ such that 
$$
\left| \Delta_{x,z} \f \right| \leq C.
$$
Moreover smooth geodesics do not exist in general and this estimate is optimal:
there exists geodesics such that $\Delta_{x,z} \f $ is bounded but not continuous.
\end{thm}

The proof of Chen's result goes by approximating the degenerate homogeneous complex Monge-Amp\`ere equation by smoother approximate geodesics
$$
(\omega+dd^c \f_\e)^{n+1}=\e dV,
$$
and establish a priori estimates as $\e>0$ decreases to zero.
We refer the reader to \cite{Bou12} for a neat presentation of this result. 

It follows from Frobenius theorem that if smooth geodesics existed, one could find a holomorphic foliation
of $X \times A$ by complex curves (the Monge-Amp\`ere foliation)
along which $\f$ is harmonic.
 When $X=\C/ \Z[ \tau]$ is an elliptic curve,
the holomorphic involution $x \mapsto f(x)=-x$ has four fixed points $p_i$
and one can then check that $\{p_i\} \times \D$ are necessarily leaves of the Monge-Amp\`ere foliation.
(assuming $\f_0,\f_1$ are $f-$invariant). Analyzing the regularity of $\f$ near such leaves allows
Lempert and Vivas to derive a contradiction in \cite{LV11}.
We refer the reader to \cite{DL12} for further improvements.

\smallskip

We will consider in the sequel generalized geodesics with lower regularity, which 
nevertheless turn out to be quite useful.

It may be also interesting to consider (weak) {\it subgeodesics}: these are paths $(\f_t(x))$ of functions
in $\H$ (or in larger  classes of $\omega$-psh functions) such that the associated 
invariant functions are $\omega$-psh functions  on $X \times A$. It follows from the computation
made in the proof of Proposition \ref{prop:Semmes} that this is equivalent to
$$
\ddot{\f}_t MA(\f_t) \geq \frac{n}{V} d \dot{\f} \wedge d^c \dot{\f}  \wedge \omega_\f^{n-1}. 
$$

\subsection{The Aubin-Mabuchi functional}

Each tangent space $T_{\f}\HH$ admits the following orthogonal decomposition
\begin{equation*}
    T_{\f}\HH = \set{\psi \in \Cinf(X) ; \: \beta_{\f}(\psi) = 0}\oplus \Real ,
\end{equation*}
where $\beta=MA$ is the 1-form defined on $\HH$ by
\begin{equation*}
    \beta_{\f}(\psi) = \int_{X}\psi \,MA(\f).
\end{equation*}
The following is a classical observation (see \cite[pp245-246]{Kol12}):

\begin{lem}
The 1-form $\beta$ is closed. 
\end{lem}


Therefore, there exists a unique function $E$ defined on the convex open set $\HH$, such that $\beta = dE$ and $E(0) = 0$. It is often called the \emph{Aubin-Mabuchi functional} and can be expressed (after integration along affine paths) by
\begin{equation}\label{equ:Aubin_Mabuchi}
    E(\f) = \frac{1}{(n+1)V_\a} \sum_{j=0}^{n} \int_{X}\f \,  
(\omega+dd^c  \f)^{j}  \wedge   \omega^{n-j}   .
\end{equation}

\begin{lem}  \label{lem:AubinMabuchi1}
The Aubin-Mabuchi functional $E$ is concave, non-decreasing and 
satisfies the cocycle condition
$$
E(\f)-E(\p)=\frac{1}{(n+1)V_\a} \sum_{j=0}^{n} \int_{X} (\f-\p) \,  
(\omega+dd^c  \f)^{j}  \wedge   (\omega+dd^c \p)^{n-j} 
$$
It is {affine} along geodesics in $\HH$.
\end{lem}

We sketch the proof for the reader's convenience.

\begin{proof}
The monotonicity property follows from the definition since the first derivative of $E$
is $dE=\beta=MA \geq 0$, a probability measure.
The concavity follows from a direct computation (see e.g. \cite{BEGZ10}), while the 
cocycle condition follows by differentiating $E(t\f+(1-t)\p)$.

The behavior of $E$ along geodesics is also easily understood:
\begin{eqnarray*}
    \frac{d^{2}}{dt^{2}} E(\f_{t}) &= & \int_{X} \ddot{\f}\,MA(\f) 
+\frac{n}{V_\a} \int_{X} \dot{\f} \, dd^c \dot{\f} \wedge \omega_\f^{n-1}  \\
&=&   \int_{X} \left\{ \ddot{\f}\,MA(\f) -\frac{n}{V_\a} d \dot{\f}  \wedge d^c \dot{\f} \wedge \omega_\f^{n-1}  \right\}=0
\end{eqnarray*}
when $\f_{t}$ is a geodesic.
\end{proof}

Since  each  path $\f_{t} = \f + t$ in $\HH$ is a geodesic, 
we obtain in particular 
\begin{equation}\label{equ:affinity}
    E(\f + t) =   E(\f)+t.
\end{equation}
Given $\f\in \HH$ there exists a unique  $c\in\Real$ such that $E(\f + c) = 0$. The restriction of the Mabuchi metric to the fiber $E^{-1}(0)$ induces a Riemannian structure on the quotient space 
$\HH_{\a} = \HH / \Real$ and allows to decompose
\begin{equation*}
    \HH = \HH_{\a} \times \Real
\end{equation*}
as a product of Riemannian manifolds.

\subsection{$\HH$ as a symmetric space}

 Consider a 2-parameters family $\f(s,t) \in \HH$ and a vector field $\psi(s,t)\in \Cinf(X)$ defined along $\f$. The curvature tensor of the Mabuchi metric on $\HH$  is defined  by
\begin{equation*}
    R_{\f}(\f_{s},\f_{t})\psi = \left(D_{s}D_{t} - D_{t}D_{s}\right)\psi,
\end{equation*}
where $\f_{s},\f_{t}$ denote the $s$ and $t$ derivatives of $\f$ and
\begin{equation*}
    D_{t} \psi = \psi_{t} + \Gamma_{\f}(\f_{t}, \psi) = \psi_{t} - <\nabla \psi , \nabla \f_{t}>_{\f}
\end{equation*}
is the covariant derivative of $\psi$.

\begin{prop}
The curvature tensor on $\HH$ can be expressed as
\begin{equation*}
    R_{\f}(\f_{s},\f_{t})\psi = - \set{\set{\f_{s},\f_{t}},\psi}
\end{equation*}
where $\set{\f_{1}, \f_{2}}$ is the \emph{Poisson bracket} associated with the symplectic structure $\omega_{\f}$. Furthermore, the covariant derivative $D_{r}R(\f_{s},\f_{t})$ of the curvature tensor $R$ vanishes.
\end{prop}

We refer the  reader to \cite[Proposition 6.27]{Kol12} for a proof of this result
which essentially says that $\HH$ is a \emph{locally symmetric space}: $\nabla R = 0$. 
In finite dimension, such a manifold is characterized by the fact that in the neighborhood of each point $x$, the (local) \emph{geodesic symmetry} of center $x$ : $\exp_{x}(X) \mapsto \exp_{x}(-X)$ is an isometry (see \cite{Hel01}).   Note however that the exponential map is not well defined in our infinite dimensional setting.

\subsection{The Mabuchi distance}

The length of an arbitrary path $(\f_t)_{t \in [0,1]}$ in $\HH$ is defined in a standard way,
$$
\ell(\f):=\int_0^1 |\dot{\f_t}| dt=\int_0^1 \sqrt{ \int_X \dot{\f}_t^2 MA({\f_t})} dt.
$$
The distance between two points in $\HH$ is then
$$
d(\f_0,\f_1):=\inf \{ \ell(\f) \, | \, \f
\text{ is a path  joining } \f_0 \text{ to } \f_1 \}.
$$

It is easy to verify that $d$ defines a semi-distance (i.e. non-negative, symmetric and satisfying the triangle inequality). It is however non trivial, in this infinite dimensional context, to check that $d$ is non degenerate (see \cite{MM05} for a striking example). This was proved by Chen in \cite{Chen00}
(see also \cite{CC02})
who established that

\begin{thm}\cite{Chen00,CC02} \label{thm:chen}
\text{ }
\begin{itemize}
\item $(\HH,d)$ is a metric space of non positive curvature;
\item the geodesics $(\f_t)$ are length minimizing;
\item any sequence of asymptotically length minimizing paths converge, in the Hausdorff topology, to the unique geodesic.
\end{itemize}
Moreover for all $t \in [0,1]$,
$$
d(\f_0,\f_1)=\sqrt{\int_X \dot{\f}_t^2 MA({\f_t})}.
$$
\end{thm}

\begin{rem}
A delicate issue here is that it is not clear that the geodesics constructed by Chen lie in $\H$.
This problem has been adressed by Berman and Demailly \cite{BD12}
(see also \cite{He12}) who showed that the whole construction
can be extended to 
$$
\H_{1,1}:=\{ \f \in PSH(X,\omega) \, | \, \Delta \f \in L^{\infty}(X) \}.
$$
\end{rem}

Observe that $d$ induces a distance on $\HH_\a$ 
(that we abusively still denote by $d$)  compatible with the riemannian splitting
$\HH=\HH_\a \times \R$, by setting
$$
d(\omega_\f,\omega_\p):=d(\f,\p)
$$
whenever the potentials $\f,\p$ of $\omega_\f,\omega_\p$ are normalized by $E(\f)=E(\p)=0$.

\medskip

It is rather easy to check that $(\HH_\a,d)$ is not a complete metric space. 
Describing the metric
completion $(\overline{\HH}_\a,d)$ is the main goal of this article. We shall always work at the level of potentials, i.e. we will try and understand the metric completion of the space $(\HH,d)$.

\section{Finite energy classes} \label{sec:energy}

We fix $\omega$ a K\"ahler form representing $\a$ and define in this section the
set ${\mathcal E}(\a)$ (resp. ${\mathcal E}^p(\a)$)
of positive closed currents $T=\omega+dd^c \f$ with full Monge-Amp\`ere mass 
(resp. finite weighted energy) in $\a$,
by defining the corresponding class ${\mathcal E}(X,\omega)$ 
(resp. ${\mathcal E}^p(X,\omega)$ ) of potentials $\f$.

\subsection{The space ${\mathcal E}(\a)$}

\subsubsection{Quasi-plurisubharmonic functions}
Recall that a function is quasi-plurisubharmonic if it is locally given as the sum of  a smooth and a psh function. In particular quasi-psh functions are upper semi-continuous and $L^1$-integrable.
Quasi-psh functions are actually in $L^p$ for all $p \geq 1$, and the induced topologies are all equivalent.
A much stronger integrability property actually holds:
Skoda's integrability theorem \cite{Sko} asserts indeed that $e^{-\e \f} \in L^1(X)$ if 
$0 < \e$ is smaller than $2/\nu(\f)$,  where $\nu(\f)$ denotes the maximal logarithmic singularity (Lelong number) of $\f$ on $X$.

Quasi-plurisubharmonic functions have gradient in $L^r$ for all $r<2$, but not in $L^2$ as shown by the local model $\log |z_1|$.

\begin{defi}
We let $PSH(X,\omega)$ denote the set of all $\omega$-plurisubharmonic functions. 
These are quasi-psh functions
$\f:X \rightarrow \R \cup \{-\infty\}$ such that
$$
\omega+dd^c \f \geq 0
$$
in the weak sense of currents. 
\end{defi}

The set $PSH(X,\omega)$ is a closed subset of $L^1(X)$, when endowed with the $L^1$-topology.

\subsubsection{Bedford-Taylor theory}

Bedford and Taylor have observed in \cite{BT82} that one can define the complex Monge-Amp\`ere operator
$$
MA(\f):=\frac{1}{V_\a} (\omega+dd^c \f)^n
$$
for all {\it bounded} $\omega$-psh function: they showed that whenever $(\f_j)$ is a sequence of smooth
$\omega$-psh functions locally decreasing to $\f$, then the smooth probability measures $MA(\f_j)$
converge, in the weak sense of Radon measures, towards a unique 
probability measure that we denote by $MA(\f)$.

At the heart of Bedford-Taylor's theory lies the following {\it maximum principle}: if $u,v$ are bounded
$\omega$-plurisubharmonic functions, then
$$
\hskip-3cm (MP) \hskip2cm 1_{\{v<u\}} MA(\max(u,v)) =1_{\{v<u\}} MA(u).
$$
This equality is elementary when $u$ is {\it continuous}, as the set $\{v<u\}$ is then a Borel open subset of $X$. When $u$ is merely {\it bounded}, this set is only open in the plurifine topology. Since Monge-Amp\`ere measures of bounded qpsh functions do not charge pluripolar sets 
(by the so called Chern-Levine-Nirenberg inequalities), and since $u$ is nevertheless {\it quasi-continuous},
this gives a heuristic justification for $(MP)$.

The reader will easily verify that the maximum principle
$(MP)$ implies the so called {\it comparison principle}:

\begin{prop}
Let $u,v$ be bounded
$\omega$-plurisubharmonic functions. Then
$$
\int_{\{v<u\}}  MA(u) \leq \int_{\{v<u\}}  MA(v).
$$
\end{prop}

\subsubsection{The class ${\mathcal E}(X,\omega)$}

Given   $\f \in PSH(X,\omega)$, we consider its canonical approximants
$$
\f_j:=\max(\f, -j) \in PSH(X,\omega) \cap L^{\infty}(X).
$$
It follows from the Bedford-Taylor theory that the measures $MA(\f_j)$ are well defined probability measures.
 Since the $\f_j$'s are decreasing, it is natural to expect that these measures converge (in the weak sense). 
The following strong monotonicity property holds:
  
\begin{lem}
The sequence
$
\mu_j:={\bf 1}_{\{ \f>-j\}} MA(\f_j)
$
is an increasing sequence of Borel measures.
\end{lem}

\noindent The proof is an elementary consequence of $(MP)$ (see \cite[p.445]{GZ07}).

\smallskip

Since the $\mu_j$'s all have total mass bounded from above by $1$ (the total mass of the measure $MA(\f_j)$),  we can consider
$$
\mu_{\f}:=\lim_{j \rightarrow +\infty} \mu_j,
$$
which is a positive Borel measure on $X$, with total mass $\leq 1$.

\begin{defi}
We set 
$$
{\mathcal E}(X,\omega):=\left\{ \f \in PSH(X,\omega) \; | \; \mu_{\f}(X)=1 \right\}.
$$
For $\f \in {\mathcal E}(X,\omega)$, we set
$
MA(\f):=\mu_{\f}.
$
\end{defi}

The notation is justified by the following important fact:
the complex Monge-Amp\`ere operator $\f \mapsto MA(\f)$ is well defined on the class 
${\mathcal E}(X,\omega)$, i.e. for every decreasing sequence of bounded (in particular 
smooth) $\omega$-psh functions $\f_j$, the probability measures 
$MA(\f_j)$ weakly converge towards $\mu_\f$, if $\f \in {\mathcal E}(X,\omega)$.

\smallskip

Every bounded $\omega$-psh function clearly belongs to ${\mathcal E}(X,\omega)$ 
since  in this case $\{\f >-j\}=X$ for $j$ large enough, hence 
$$
\mu_{\f} \equiv \mu_j=MA(\f_j)=MA(\f).
$$
The class ${\mathcal E}(X,\omega)$ also contains many $\omega$-psh functions which are unbounded.
When $X$ is a compact Riemann surface ($n=\dim_{\C} X=1$),
the set ${\mathcal E}(X,\om)$ is the set of $\om$-sh functions
whose Laplacian does not charge polar sets.

\begin{rem}
If $\f \in PSH(X,\omega)$ is normalized so that $\f \leq -1$, then $-(-\f)^\e $ belongs to 
${\mathcal E}(X,\om)$ whenever $0 \leq \e <1$.
The functions which belong to the class ${\mathcal E}(X,\om)$,
although usually unbounded,
have relatively mild singularities. In particular they have zero Lelong numbers.
\end{rem}

It is shown in \cite{GZ07} that the maximum principle $(MP)$ and the comparison principle
continue to hold in the class ${\mathcal E}(X,\omega)$. The latter can be characterized as the largest class for which the complex Monge-Amp\`ere is well defined and the maximum principle holds. We further note
that the {\it domination principle} holds:

\begin{prop} \label{pro:dp}
If $\f,\p \in {\mathcal E}(X,\om)$ are such that 
$$
\f(x) \leq \p(x)
\text{ for } MA(\p)-\text{a.e. } x,
$$
then $\f(x) \leq \p(x)$ for all $x \in X$.
\end{prop}

Let us stress that the convergence of the canonical approximating measures
$\mu_j=MA(\max(\f,-j))$ towards $\mu_{\f}$ holds in the (strong) sense of Borel measures, i.e. for all Borel sets $B$,
$$
\mu_{\f}(B):=\lim_{j \rightarrow +\infty} \mu_j(B).
$$
In particular when $B=P$ is a pluripolar set, we obtain $\mu_j(P)=0$, hence 
$$
\mu_{\f}(P)=0
\text{ for all pluripolar sets } P.
$$
Conversely, one can show \cite{GZ07,BEGZ10} that a probability measure $\mu$ equals $\mu_{\f}$ for some
$\f \in {\mathcal E}(X,\omega)$ whenever $\mu$ does not charge pluripolar sets
(one then says that $\mu$ is non-pluripolar).

\smallskip

It follows from the $\partial\overline{\partial}$-lemma that any positive closed current 
$T \in \a$ writes $T=\omega+dd^c \f$ for some function $\f \in PSH(X,\omega)$
which is unique up to an additive constant. 

\begin{defi}
We let ${\mathcal E}(\a)$ denote the set of all positive currents in $\a$,
$T=\omega+dd^c \f$,
with $\f \in {\mathcal E}(X,\omega)$.
\end{defi}

The definition is clearly independent of the choice of the potential $\f$.

\subsection{The complete metric spaces ${\mathcal E}^p(\a)$}

\subsubsection{Weighted energy classes}

Let $\chi:\R^- \rightarrow \R^-$ be an increasing function
such that $\chi(-\infty)=-\infty$. Following \cite{GZ07} we let $\W$ denote the set of all such weights,
and set
$$
\W^\pm:=\{ \chi \in \W \, | \, \mp\chi \text{ is convex} \}
$$

\begin{defi}
We let
${\mathcal E}_{\chi}(X,\om)$ denote the set of 
$\om$-psh
functions with
finite $\chi$-energy, i.e. 
$$
{\mathcal E}_{\chi}(X,\om):=\left\{ \f \in {\mathcal E}(X,\om) \, / \, 
 \chi(-|\f|) \in L^1( MA(\f)) \right\}.
$$
When $\chi(t)=-(-t)^p$, $p>0$, we set $\E^p(X,\omega)=\E_\chi(X,\omega)$. We let
$$
{\mathcal E}_{\chi}(\a)=\{ T=\omega+dd^c \f \, | \, \f \in {\mathcal E}_{\chi}(X,\omega) \}
\; \; \text{ and } \; \; 
{\mathcal E}^p(\a)
$$
denote the corresponding sets of finite weighted energy currents.
\end{defi}

We list a few important properties of these classes and refer the reader to \cite{GZ07} for the proofs:
\begin{itemize}
\item $\E(X,\omega)=\cup_{\chi \in \W} \, \E_\chi(X,\omega)=\cup_{\chi \in \W^-} \, \E_\chi(X,\omega)$;
\item $PSH(X,\om) \cap L^{\infty}(X)=\cap_{\chi \in {\mathcal W}} \E_\chi(X,\omega)=
\cap_{\chi \in {\mathcal W}^+} \E_\chi(X,\omega)$;
\item when $\chi \in {\mathcal W}^+$ any $\f \in {\mathcal E}_{\chi}(X,\om)$ is such that
$\nabla_\omega \f \in L^2(\om^n)$;
\item $\f \in \E^p(X,\omega)$ if and only if for any (resp. one)  sequence of bounded
$\omega$-functions decreasing to $\f$, 
$\sup_j \int_X (-\f_j)^p MA(\f_j) <+\infty$.
\end{itemize}

 \begin{prop} \cite[Proposition 3.8]{GZ07} \label{pro:GZ3.8}
Fix $p>0$.
There exists $C_{p}>0$ such that for all 
$0 \geq  \f_0,\ldots,\f_n \in PSH(X,\om) \cap L^{\infty}(X)$,
$$
0 \leq \int_X (-\f_0)^p \, \om_{\f_1} \wedge \cdots \wedge \om_{\f_n}
\leq C_{p} \max_{0 \leq j \leq n} \left[ \int_X (-\f_j)^p \,  \om_{\f_j}^n \right].
$$
In particular the class $\E^p(X,\omega)$ is starshaped and convex.
\end{prop}

\subsubsection{Strong topology on  ${\mathcal E}^1(\a)$}

The class $\E^1(X,\omega)$ plays a central role. Its defining weight 
$\chi(t)=t$ is the only weight which is both convex and 
concave\footnote{Up to scaling and translating, two operations which leave the class invariant.},
$$
\E^1(X,\omega)=\cup_{\chi \in \W^-} \, \E_\chi(X,\omega)  \bigcap 
\cup_{\chi \in {\mathcal W}^+} \E_\chi(X,\omega).
$$
If we extend the Aubin-Mabuchi functional using its monotonicity property,
$$
E(\f):=\inf \{ E(\p) \, | \, \p \in PSH(X,\omega) \cap {\mathcal C}^{\infty}(X) \text{ and } \f \leq\p \},
$$
then 
$$
\E^1(X,\omega)=\{ \f \in PSH(X,\omega) \, | \, E(\f)>-\infty \}.
$$
Thus $\E^1(X,\omega)$ is the natural frame for the variational approach to studying complex 
Monge-Amp\`ere operators \cite{BBGZ}. Set
$$
I(\f,\p)=\int_X (\f-\p) \left( MA(\p)-MA(\f) \right).
$$

It has been shown in \cite{BBEGZ} that $I$ defines a complete metrizable uniform structure on
${\mathcal E}^1(\a)$. More precisely we identify ${\mathcal E}^1(\a)$ with the set
$$
\E^1_{norm}(X,\omega)=\{ \f \in \E^1(X,\omega) \, | \, \sup_X \f=0 \}
$$
of normalized potentials. Then
\begin{itemize}
\item $I$ is symmetric and positive on $\E^1_{norm}(X,\omega)^2 \setminus \{ \rm{diagonal} \}$;
\item $I$ satisfies a quasi-triangle inequality \cite[Theorem 1.8]{BBEGZ};
\item $I$ induces a uniform structure which is metrizable \cite{Bourbaki};
\item the metric space $(\E^1(\a),d_1)$ is complete \cite[Proposition 2.4]{BBEGZ}.
\end{itemize}

\begin{defi}
The {\it strong topology} on $\E^1(\a)$ is the topology defined by $I$. We fix in the sequel a distance 
$d_1$ which defines this metrizable topology. 
\end{defi}

The corresponding notion of convergence corresponds 
to the {\it convergence in energy} previously introduced
in \cite{BBGZ} (see \cite[Proposition 2.3]{BBEGZ}). It is the coarsest refinement of the weak topology
such that $E$ becomes continuous. In particular 
$$
\text{ if } T_j \longrightarrow T  \text{ in } (\E^1(\a),d_1),  \; \text{ then } \; 
T_j^n \longrightarrow T^n
$$ 
in the weak sense of Radon measures, while the Monge-Amp\`ere operator is usually discontinuous for the weak topology of currents.

\begin{exa}
When $\dim_\C X=n=1$, $\E^1(X,\omega)=PSH(X,\omega) \cap W^{1,2}(X)$ is the set of 
$\omega$-subharmonic functions with square integrable gradient. The strong topology on
$\E^1(\a)$ is the one induced by the Sobolev norm. 

When $X=\P^1_\C$ is the Riemann sphere, $\omega=\omega_{FS}$ is the Fubini-Study K\"ahler form,  
and $\e_j \searrow 0^+$, $C_j \nearrow +\infty$, the functions
$$
\f_j[z]=\e_j \max ( \log ||z||, \log |z_0|-C_j)-\e_j \log ||z|| \in \E^1(X,\omega)
$$
converge to zero in $L^2(X)$ but not in the strong topology if $\e_j^2 C_j \geq 1$ since
$$
I(\f_j,0)=\int_{\P^1} d \f_j \wedge d^c \f_j \sim \e_j^2 C_j.
$$
\end{exa}

\subsubsection{Strong topology on  ${\mathcal E}^2(\a)$}

We now introduce an analogous strong topology on the class ${\mathcal E}^2(\a)$.
For $\f,\p \in \E^2(X,\omega)$, we set
$$
I_2(\f,\p):=\sqrt{ \int_X (\f-\p)^2 \left[ \frac{MA(\f)+MA(\p)}{2} \right]}
$$

This quantity is well-defined (and finite) by \cite[Proposition 3.6]{GZ07}. It is obviously non-negative and 
symmetric. It follows from the {\it domination principle} (Proposition \ref{pro:dp}) that
$$
I_2(\f,\p)=0 \Longrightarrow \f=\p.
$$  

\begin{defi}
The strong topology on $\E^2(\a)$ is the one induced by $I_2$. 
\end{defi}

It follows from Propositions \ref{prop:lowerbound} and \ref{prop:upperbound} below that the Mabuchi topology (induced by the Mabuchi distance) is stronger than the strong topology on $\E^1(\a)$
and weaker than the strong topology on $\E^2(\a)$. We set
$$
\E^2_{norm}(X,\omega)=\{ \f \in \E^2(X,\omega) \; | \; \sup_X \f=0 \}
$$

Note that if a sequence $(\f_j) \in \E^2_{norm}(X,\omega)$ is a Cauchy sequence for $I_2$, then it is
a Cauchy sequence in $(\E^1_{norm}(X,\omega),d_1)$ since
$$
0 \leq I(\f,\p)=\int_X (\f-\p) \left[ MA(\p)-MA(\f) \right]\leq \sqrt{2} I_2(\f,\p),
$$
as follows from the Cauchy-Schwarz inequality.
Since $(\E^1_{norm}(X,\omega),d_1)$ is complete, we infer the existence of $\f \in \E^1_{norm}(X,\omega)$
such that $d_1(\f_j,\f) \rightarrow 0$.

Now $I_2(\f_j,0)$ is bounded and $MA(\f_j)$ weakly converges to $MA(\f)$
(by \cite[Proposition 1.6]{BBEGZ}). It follows therefore from Fatou's and Hartogs' lemma that 
$\f \in \E^2_{norm}(X,\omega)$.

One would now like to prove that $I_2(\f_j,\f) \rightarrow 0$ and conclude that the space 
$(\E^2_{norm}(X,\omega),I_2)$ is complete,
arguing as in \cite[Proposition 2.4]{BBEGZ}.
There is an abuse of terminology here as we haven't  checked that $I_2$ induces a uniform structure.
This would easily follow from Proposition \ref{prop:python} if one could
show that $I_2$ satisfies a quasi-triangle inequality 
(like $I$ does, see \cite[Theorem 1.8]{BBEGZ}).

\begin{lem} \label{lem:controleL2}
Let $\f,\p$ be bounded $\omega$-psh functions and $S$ be a positive closed current of bidimension $(1,1)$
on $X$. If $\f \leq \p$, then
$$
\int_X (\f-\p)^2 \omega_\p \wedge S \leq \int_X (\f-\p)^2 \omega_\f \wedge S.
$$
In particular for all $0 \leq j \leq n$,
$$
V_\a^{-1} \int_X (\f-\p)^2 \omega_\p^j \wedge \omega_\f^{n-j}
\leq \int_X (\f-\p)^2 MA(\f).
$$
\end{lem}

\begin{proof}
By Stokes theorem,
$$
\int_X (\f-\p)^2 \omega_\f \wedge S-\int_X (\f-\p)^2 \omega_\p \wedge S
=2\int_X (\p-\f) d(\f-\p) \wedge d^c (\f-\p) \wedge S
$$
is non-negative if $(\p-\f) \geq 0$. 

The second assertion follows by applying the first one  inductively.
\end{proof}

\begin{prop} \label{prop:python}
For $\f,\p \in \E^2(X,\omega)$,
$$
I_2(\f,\p)^2=I_2(\f,\max(\f,\p))^2+I_2(\max(\f,\p),\p)^2.
$$
\end{prop}

\begin{proof}
Recall that the maximum principle insures that
$$
\1_{\{ \f<\p\}} MA(\max(\f,\p))= \1_{\{ \f<\p\}}MA(\p),
$$
while $(\f-\max(\f,\p))^2=0$ on $(\f \geq \p)$, thus
$$
I_2(\f,\max(\f,\p))^2=\frac{1}{2}\int_{\{\f<\p \}} (\f-\p)^2 \left[ MA(\f)+MA(\p)\right].
$$
Similarly
$$
I_2(\p,\max(\f,\p))^2=\frac{1}{2}\int_{\{\f>\p \}} (\f-\p)^2 \left[ MA(\f)+MA(\p)\right]
$$
and the result follows since 
$$
I_2(\f,\p)^2=\frac{1}{2}\int_{\{\f \neq \p \}} (\f-\p)^2 \left[ MA(\f)+MA(\p)\right].
$$
\end{proof}

\subsection{Comparing distances}

\subsubsection{Comparing $d$ and $I$}

The following lower bound is the first key observation for understanding the
metric completion of $(\HH,d)$:

\begin{prop} \label{prop:lowerbound}
For all $\f_0,\f_1 \in \HH$, 
\begin{eqnarray*}
\frac{1}{n+1} I(\f_0,\f_1) \! \! \! &+&  \! \! \!
\max\left\{ \int_X (\f_0-\f_1) MA(\f_0) ;  \int_X (\f_1-\f_0) MA(\f_1) \right\} \\
& \leq&  d(\f_0,\f_1).
\end{eqnarray*}
\end{prop}

\begin{proof}
The cocycle condition (see Lemma \ref{lem:AubinMabuchi1}) yields
\begin{eqnarray*}
\lefteqn{E(\f_0)-E(\f_1)-\int_X (\f_0-\f_1) MA(\f_0) }\\
&=&\frac{1}{(n+1)V_\a} \sum_{j=1}^n \sum_{\ell=0}^{j-1}
\int_X d(\f_1-\f_0) \wedge d^c (\f_1-\f_0) \wedge \omega_{\f_1}^\ell \wedge
\omega_{\f_0}^{n-\ell-1}  \\
&=& \frac{1}{(n+1)V_\a} \sum_{\ell=0}^{n-1} (n-\ell)
\int_X d(\f_1-\f_0) \wedge d^c (\f_1-\f_0) \wedge \omega_{\f_1}^\ell \wedge
\omega_{\f_0}^{n-\ell-1} 
\end{eqnarray*}
while
\begin{eqnarray*}
I(\f_0,\f_1)&=&\int_X (\f_0-\f_1) \left(MA(\f_1)-MA(\f_0) \right) \\
&=& \frac{1}{V_\a} \sum_{\ell=0}^{n-1}
 \int_X d(\f_1-\f_0) \wedge d^c (\f_1-\f_0) \wedge \omega_{\f_1}^\ell \wedge
\omega_{\f_0}^{n-\ell-1},
\end{eqnarray*}
thus
$$
\frac{1}{n+1}I(\f_0,\f_1) +\int_X (\f_0-\f_1) MA(\f_0) \leq E(\f_0)-E(\f_1).
$$

Let $(\f_t)_{t \in [0,1]}$ be the geodesic joining $\f_0$ to $\f_1$.
Recall from Lemma \ref{lem:AubinMabuchi1} that $t \mapsto E(\f_t)$ is affine.
Since $E$ is a primitive of the (normalized) complex Monge-Amp\`ere, we infer
$$
E(\f_1)-E(\f_0)=\int_0^1 \int_X \dot{\f_t} MA(\f_t) dt.
$$
By Chen's theorem the  length of this geodesic equals $d(\f_0,\f_1)$, while by definition
the Riemannian energy 
$$
t \mapsto \int_X (\dot{\f_t})^2 MA(\f_t)
$$
is constant hence equals $d(\f_0,\f_1)^2$. 
Cauchy-Schwarz inequality thus yields
$$
 E(\f_0)-E(\f_1) \leq \sqrt{\int_X (\dot{\f_t})^2 MA(\f_t)}
=d(\f_0,\f_1),
$$
proving that
$$
\frac{1}{n+1} I(\f_0,\f_1) +\int_X (\f_0-\f_1) MA(\f_0) 
 \leq d(\f_0,\f_1).
$$
The desired bound follows by reversing the roles of $\f_0$ and $\f_1$.
\end{proof}

\subsubsection{Comparing $d$ and $I_2$}

\begin{prop} \label{prop:upperbound}
For all $\f_0,\f_1 \in \HH$, 
$$
0 \leq d(\f_0,\f_1)  \leq 2 I_2(\f_0,\f_1).
$$
Moreover if $\f_0 \leq \f_1$ then $ I_2(\f_0,\f_1) \leq \sqrt{ \int_X (\f_1-\f_0)^2 MA(\f_0)}$ and
$$
 d(\f_0,\f_1)  \leq  \sqrt{ \int_X (\f_1-\f_0)^2 MA(\f_0)}  \leq   2^{2+n/2} d(\f_0,\f_1).
$$
\end{prop}

The last upper-bound uses a nice trick due to Darvas \cite{Dar14} as we explain below.
It would be very useful if this upper bound would hold without any restriction, thus showing that
the topology induced by $I_2$ and $d$ are equivalent.

\begin{proof} 
We first assume that $\f_0 \leq \f_1$. 
The inequality 
$$
 I_2(\f_0,\f_1) \leq \sqrt{ \int_X (\f_1-\f_0)^2 MA(\f_0)}
$$ 
follows from Lemma \ref{lem:controleL2}.
Let 
$(\f_t)$ be the geodesic joining $\f_0$ to $\f_1$. It follows from the maximum principle that
$t \mapsto \f_t$ is increasing. As it is convex as well, we infer
$$
0 \leq \dot{\f}_0 \leq \f_1-\f_0 \leq \dot{\f}_1
$$
hence 
$$
\int_X (\f_1-\f_0)^2 MA(\f_1) \leq \int_X (\dot{\f}_1)^2 MA(\f_1) =
d(\f_0,\f_1)^2
$$
and similarly $d(\f_0,\f_1)^2 \leq \int_X (\f_1-\f_0)^2 MA(\f_0)$.

We give an alternative proof of this upper bound which does not use Chen's result. This might be of interest in more singular contexts.
We can join $\f_0$ to $\f_1$ by a straight line $\p_t=t\f_1+(1-t)\f_0$, thus
\begin{eqnarray*}
d(\f_0,\f_1) &\leq& \ell(\f)=\int_0^1 \sqrt{\int_X (\f_1-\f_0)^2 MA(\f_t)} dt \\
&\leq& \sqrt{\int_0^1\int_X (\f_1-\f_0)^2 MA(\f_t) dt},
\end{eqnarray*}
by Cauchy-Schwarz inequality. Now
$$
MA(\f_t)=V_\a^{-1} \sum_{j=0}^n \left( \begin{array}{c} n \\ j \end{array} \right) t^j (1-t)^{n-j} 
\omega_{\f_1}^j \wedge \omega_{\f_0}^{n-j}
$$
and for all $0 \leq j \leq n$,
$$
\int_0^1 t^j (1-t)^{n-j} dt=\frac{\left( \begin{array}{c} n \\ j \end{array} \right)^{-1}}{n+1},
$$
hence 
$$
 \frac{1}{(n+1)V_\a} \sum_{j=0}^n  
\int_X (\f_1-\f_0)^2 \omega_{\f_1}^j \wedge \omega_{\f_0}^{n-j} \leq \int_X (\f_1-\f_0)^2 MA(\f_0),
$$
as follows from Lemma \ref{lem:controleL2}, yielding
$$
d(\f_0,\f_1) \leq \sqrt{\int_X (\f_1-\f_0)^2 MA(\f_0)}.
$$

\smallskip

We now show that 
$ \int_X (\f_1-\f_0)^2 MA(\f_0)   \leq   2^{n+4} d(\f_0,\f_1)^2$.
Observe that $\frac{\f_0+\f_1}{2} \in \H_\omega$ with 
$$
 MA(\f_0)   \leq  2^n \, MA\left( \frac{\f_0+\f_1}{2} \right) 
$$ 
hence 
\begin{eqnarray*}
\int_X (\f_0-\f_1)^2 MA(\f_0) &=&4 \int_X \left(\f_0-\frac{\f_0+\f_1}{2} \right)^2 MA(\f_0) \\
& \leq & 2^{n+2}\int_X   \left(\f_0-\frac{\f_0+\f_1}{2} \right)^2  MA\left( \frac{\f_0+\f_1}{2} \right)  \\
& \leq & 2^{n+2} d\left(\f_0,\frac{\f_0+\f_1}{2} \right)^2,
\end{eqnarray*}
as follows from the first step of the proof since  $\f_0 \leq   \f_1 $.
The triangle inequality and Lemma \ref{lem:dar} below yield
$$
d\left(\f_0,\frac{\f_0+\f_1}{2} \right) \leq d(\f_0,\f_1)+d\left(\f_1,\frac{\f_0+\f_1}{2} \right) 
\leq 2 \, d(\f_0, \f_1)
$$
hence 
$$
\int_X (\f_0-\f_1)^2 MA(\f_0) \leq    2^{n+4} d(\f_0,0)^2.
$$

\smallskip

We finally treat the first upper bound of the Proposition which does not require $\f_0$ to lie below $\f_1$. It follows from the triangle inequality that
\begin{eqnarray*}
\lefteqn{d(\f_0,\f_1) \leq  d(\f_0,\max(\f_0,\f_1))+d(\max(\f_0,\f_1), \f_1) }\\
&\leq & \sqrt{\int_{\{\f_0<\f_1\}} (\f_1-\f_0)^2 MA(\f_0)}+
\sqrt{\int_{\{\f_0>\f_1\}} (\f_1-\f_0)^2 MA(\f_1)} \\
&\leq& \sqrt{2} \, \sqrt{ \int_X (\f_1-\f_0)^2 \left[ MA(\f_0)+MA(\f_1)\right]} ,
\end{eqnarray*}
by using the elementary inequality $\sqrt{a}+\sqrt{b} \leq \sqrt{2}\sqrt{a+b}$. 
\end{proof}

The following observation is due to Darvas \cite{Dar14}:

\begin{lem} \label{lem:dar}
 Assume $\f_0,\f_1 \in \H$ satisfy $\f_0 \leq \f_1$. Then
$$
d\left(\f_1, \frac{\f_0+\f_1}{2} \right)  \leq  d(\f_0, \f_1).
$$
\end{lem}

We include a proof for the convenience of the reader.

\begin{proof}
Let $\f_t$ (resp. $\p_t$) denote the geodesic joining $\f_0$ (resp.  $(\f_0+\f_1)/2$) to $\f_1$ .
Since $\f_0 \leq \f_1 $, it follows from the maximum principle that
$t \mapsto \f_t$, $t \mapsto \p_t$ are non decreasing and
$
{\f_t}\leq \p_t 
$
hence
$$
\frac{\f_t-\f_1}{t-1} \geq \frac{\p_t-\p_1}{t-1}
$$
since $\f_1=\p_1$.
Therefore
$$
\dot{\f}_1 \geq \dot{\p}_1 \geq 0
$$
and we infer
$$
 \int_X  (\dot{\p}_1)^2 MA(\p_1)=d\left(\f_1, \frac{\f_0+\f_1}{2} \right)^2  \leq  d(\f_0, \f_1)^2
 =\int_X  (\dot{\f}_1)^2 MA(\f_1).
$$
\end{proof}

\begin{rem}
The lower bound $\sqrt{ \int_X (\f_1-\f_0)^2 MA(\f_1)} \leq d(\f_0,\f_1) $ when 
$\f_0 \leq \f_1$ is a particular case of the lower bound obtained by Blocki in \cite{Blo12},
who showed -generalizing a previous lower bound due to Donaldson \cite{Don99} and Chen \cite{Chen00}-
that
$$
\sqrt{ \int_X (\f_1-\f_0)^2 MA(\max(\f_0,\f_1))}  \leq d(\f_0,\f_1).
$$

Indeed since the measure $(\f_1-\f_0)^2 MA(\max(\f_0,\f_1))$ does not charge the set 
$\{\f_0=\f_1\}$, and since $MA(\max(\f_0,\f_1))=MA(\f_0)$ on the Borel set 
$\{\f_0>\f_1\}$, it follows from the maximum principle $(MP)$ that
\begin{eqnarray*}
\lefteqn{\int_X (\f_1-\f_0)^2 MA(\max(\f_0,\f_1))} \\
&&=\int_{(\f_0>\f_1)} (\f_1-\f_0)^2 MA(\f_0)+\int_{(\f_0<\f_1)} (\f_1-\f_0)^2 MA(\f_1) \\
&& \geq \max \left\{ \int_{(\f_0>\f_1)} (\f_1-\f_0)^2 MA(\f_0); \int_{(\f_0<\f_1)} (\f_1-\f_0)^2 MA(\f_1) \right\},
\end{eqnarray*}
which yields the contents of \cite[Theorem 1.2]{Blo12}.
\end{rem}

\section{Weak geodesics} \label{sec:geod}

\subsection{Kiselman transform and geodesics}

For $\f,\p \in \H$ we let $\f \vee \p$ denote the greatest $\omega$-psh function that lies below
$\f$ and $\p$. In the notations of Berman-Demailly \cite{BD12}
$$
\f \vee \p=P(\min(\f,\p)),
$$
while $\f \vee \p$ is denoted by $P(\f,\p)$ in \cite{Dar14}.

It follows from the work of Berman and Demailly that the function
$\f \vee \p$ has bounded Laplacian (i.e. $\f \vee \p \in \H_{1,1}$) and that its Monge-Amp\`ere measure is supported on
the coincidence set
$$
\{ x \in X \, | \, \f \vee \p(x)=\min(\f,\p)(x) \}.
$$
Let us stress that Theorem \ref{thm:chen} continues to hold when $\f,\p$ (hence $\f \vee \p$) belong to
the set $\H_{1,1}$.

\smallskip

An important consequence of Kiselman minimum principle \cite{Kis78}
is the following key observation due to  Darvas and Rubinstein \cite{Dar14}:

\begin{lem} \label{lem:kis}
Let $(\f_t)$ be the geodesic joining $\f=\f_0 \in \H$ and $\p=\f_1 \in \H$.
Then for all $x \in X$,
$$
\f \vee \p(x)=\inf_{t \in [0,1]} \f_t(x).
$$
\end{lem}

\begin{proof}
It follows from Kiselman minimum principle \cite{Kis78} that 
$$
x \mapsto \inf_{t \in [0,1]} \f_t(x)
$$ 
is a $\omega$-psh function.
Since it lies both below $\f_0$ and $\f_1$, we infer 
$$
\inf_{t \in [0,1]} \f_t \leq \f \vee \p.
$$

Conversely $(t,x) \mapsto \f \vee \p(x)$ is a subgeodesic (independent of $t$), hence
for all $t,x$, $ \f \vee \p(x) \leq \f_t (x)$ and the result follows.
\end{proof}

We now establish a very useful relation due to Darvas \cite{Dar14}:

\begin{prop} \label{prop:pyth}
Assume $\f,\p \in \H$. Then
$$
d^2(\f,\p)=d^2(\f,\f \vee \p)+d^2(\f \vee \p, \p).
$$
\end{prop}

We reproduce the proof of Darvas for the convenience of the reader.

\begin{proof}
Recall from Theorem \ref{thm:chen} (or rather its extension to $\H_{1,1}$) that
$$
d^2(\f,\p)=\int_X (\dot{\f_0})^2 MA(\f_0),
$$
where $(\f_t)$ denotes the geodesic from $\f=\f_0 \in \H$ to $\p=\f_1 \in \H$. 

Observe that for all $s \in \R$,
$$
\{ x \in X \, | \, \dot{\f_0}(x) \geq s \} =\{ x \in X \, | \,  \f_0 \vee (\f_1-s)(x)=\f_0(x)\}.
$$
It suffices to check this relation for $s=0$ and then apply it to the geodesic $(t,x) \mapsto \f_t(x)-s t$
joining $\f_0$ to $\f_1-s$. Now Lemma \ref{lem:kis} shows that
$$
\f_0 \vee \f_1(x)=\inf_{t \in [0,1]} \f_t(x)
$$
hence $\f_0 \vee \f_1(x)=\f_0(x)$ iff $\f_t(x) \geq \f_0(x)$ for all $t \in [0,1]$
 iff $\dot{\f}_0(x) \geq  0$, since $t \mapsto \f_t(x)$ is convex.
We infer that
$$
\int_{ \{ \dot{\f_0}>0\}} (\dot{\f_0})^2 MA(\f_0)
=2 \int_0^{+\infty} s \, MA(\f_0)(\{   \f_0 \vee (\f_1-s)=\f_0 \}) ds.
$$

Let now $(u_t)$ denote the geodesic joining $\f_0 \vee \f_1$ to $\f_1$. It is non decreasing
in $t$ since $\f_0 \vee \f_1 \leq \f_1$. Thus
\begin{eqnarray*}
\lefteqn{d^2(\f_0 \vee \f_1,\f_1) =\int_X (\dot{u_0})^2 MA(\f_0 \vee \f_1) } \\
&=& 2 \int_0^{+\infty} s \, 
MA(\f_0 \vee \f_1)(\{   (\f_0 \vee \f_1)  \vee (\f_1-s)=\f_0 \vee \f_1 \}) ds \\
&=& 2 \int_0^{+\infty} s \, 
MA(\f_0 \vee \f_1)(\{   \f_0  \vee (\f_1-s)=\f_0 \vee \f_1 \}) ds \\
&=& 2 \int_0^{+\infty} s \, 
MA(\f_0)(\{   \f_0  \vee (\f_1-s)=\f_0  \}) ds \\
&=& \int_{ \{ \dot{\f_0}>0\}} (\dot{\f_0})^2 MA(\f_0),
\end{eqnarray*}
using that $MA(\f_0 \vee \f_1)$ is supported on the contact set 
$\{ \f_0 \vee \f_1= \min(\f_0,\f_1)\}$ and equals $MA(\f_0)$ on this set whenever
$\f_1<\f_0$.

One shows similarly that 
$$
d^2(\f_0, \f_0 \vee \f_1) =\int_{ \{ \dot{\f_0} < 0\}} (\dot{\f_0})^2 MA(\f_0)
$$
and the desired identity follows.
\end{proof}

We note for later use the following consequence:

\begin{cor} \label{cor:min}
If  $\f,\p \in \H$ then
$$
d(\f,\f \vee \p) \leq d(\f,\p).
$$
and
$$
I_2(\f,\f \vee \p)^2+I_2(\f \vee \p,\p)^2 \leq 2^{n+4} d(\f,\p)^2.
$$
\end{cor}

This follows  from previous result together with Proposition \ref{prop:upperbound}, since
$\f \vee \p \leq \f$ and $\f \vee \p \leq \p$.

\subsection{Metric completion}

For $\f,\p \in \E^2(X,\omega)$ we let $\f_j,\p_k$ denote sequences of elements in $\H$
decreasing to $\f,\p$ respectively, and set
$$
D(\f,\p):=\liminf_{j,k \rightarrow +\infty} d(\f_j,\p_k).
$$
We list in the proposition below various properties of this extension.

\begin{prop}
\text{ }

i) $D$ is a distance on $\E^2(X,\omega)$
which coincides with $d$ on $\H$;

ii) the definition of $D$ is independent of the choice of the approximants;

iii) $D$ is continuous along decreasing sequences in $\E^2(X,\omega)$.

\smallskip Moreover all previous inequalities comparing  $d$ and $I_2$ on $\H$
extend to inequalities between $D$ and $I_2$ on $\E^2(X,\omega)$.
\end{prop}

In the sequel we will therefore denote $D$ by $d$.

\begin{proof}
It is a tedious exercise to verify that $D$ defines a "semi-distance", i.e. satisfies all properties of a distance
but for the separation property. The latter will be a consequence of the last statement in the proposition: if $D(\f,\p)=0$, it follows from Corollary \ref{cor:min} that 
$$
I_2(\f,\f \vee \p)=I_2(\p, \f \vee \p)=0,
$$
hence the domination principle yields $\f=\f \vee \p=\p$.

One can check that $D$ coincides with $d$ on $\H$ in various ways.
It follows from Proposition \ref{prop:Bo} that
$$
d(\f,\f_j) \leq \| \f-\f_j \|_{L^{\infty}(X)} \rightarrow 0
$$
and similarly $d(\p,\p_k) \rightarrow 0$ hence the triangle inequality yields $d(\f,\p)=D(\f,\p)$.
Using ii) one can also use the constant sequences $\f_j \equiv \f$ and $\p_k \equiv \p$
to obtain this equality.

We now prove ii). Let $\f_j, u_j$ (resp. $\p_k, v_k$) denote two sequences of
elements of $\H$ decreasing to $\f$ (resp. $\p$). We can assume without loss of generality
that these sequences are intertwining,
i.e. for all $j,k \in \N$, there exists $p,q \in \N$ such that
$$
\f_j \leq u_{p}
\; \text{ and }  \; 
\p_k \leq v_{q},
$$
with similar reverse inequalities. It follows therefore from Proposition \ref{prop:upperbound}
and the triangle inequality that
\begin{eqnarray*}
\left| d(\f_j,\p_k)-d(u_p,v_q) \right| &\leq& 
d(\f_j,u_p)+d(\p_k,v_q) \\
&\leq&  2 I_2(\f_j,u_p)+2 I_2(\p_k,v_q).
\end{eqnarray*}
Now
$$
I_2(\f_j,u_p) \leq \sqrt{  \int_X (\f_j-u_p)^2 MA(\f_j)}  \leq \sqrt{ 3^n \int_X (\f-u_p)^2 MA(\f)}  
$$
 as follows from \cite[Lemma 3.5]{GZ07}. 
The monotone convergence theorem therefore yields  $I_2(\f_j,u_p) \rightarrow 0$ 
as $p \rightarrow +\infty$ and similarly $ I_2(\p_k,v_q) \rightarrow 0$ as $q \rightarrow +\infty$,
proving ii).

One shows iii) with similar arguments. The extension of the inequalities comparing $d$ and 
$I_2$ (or $d$ and $I$) follows from \cite[Theorem 2.17]{BEGZ10}.
\end{proof}

 \begin{prop}
The metric space $(\E_{norm}^2(X,\omega),d)$ is complete. Moreover the Mabuchi topology dominates the
topology induced by $I$: if  a sequence converges for the Mabuchi distance, then it converges in energy.
\end{prop}

 \begin{proof}
Let $(\f_j) \in \E_{norm}^2(X,\omega)^{\N}$ be a Cauchy sequence for $d$. Since 
$\sup_X \f_j$ is bounded, the sequence is relatively compact for the (weak) $L^1$-topology.
Let $\p$ be a cluster point for the $L^1$-topology. We claim that $\p \in \E_{norm}^2(X,\omega)$,
$$
d(\f_j,\p) \rightarrow 0
\text{ and }
I(\p,\f_j) \rightarrow 0.
$$

Extracting and relabelling, we can assume that 
$$
\f_j \stackrel{L^1}{\longrightarrow}  \p
\; \; \text{ and }  \; \; 
d(\f_j,\f_{j+1}) \leq 2^{-j}.
$$
Set $\f_{-1} \equiv 0$ and for $k \geq j$,
$
\p_{j,k}:=\f_j \vee \f_{j+1} \vee \cdots \vee \f_k.
$
Observe that 
\begin{eqnarray*}
d(0,\p_{j,k}) &\leq& \sum_{\ell=-1}^{j-1} d(\f_\ell,\f_{\ell+1}) +d( \f_j, \p_{j,k}) \\
&\leq & \sum_{\ell=-1}^{j} d(\f_\ell,\f_{\ell+1}) +d( \f_{j+1}, \p_{j+1,k}) \leq 2,
\end{eqnarray*}
as
$$
d( \f_j, \p_{j,k})=d( \f_j, \f_{j} \vee \p_{j+1,k}) \leq d( \f_j,  \p_{j+1,k})
\leq 2^{-j} +d( \f_{j+1}, \p_{j+1,k}).
$$
Since $\p_{j,k} \leq 0$ we obtain the uniform bound
$$
I_2(0,\p_{j,k})) \leq 2^{3+n/2}
$$
showing that $\p_j:=\lim_{k \rightarrow +\infty} \p_{j,k}) \in \E_{norm}^2(X,\omega)$.
Now $\p_j$ increases a.e. towards $\p$, hence $\p \in \E_{norm}^2(X,\omega)$ and 
\cite[Theorem 2.17]{BEGZ10} yields
$$
I(\p,\p_j)+I_2(\p_j,\p) \longrightarrow  0.
$$

It follows therefore from Proposition \ref{prop:upperbound} that
$d(\p,\p_j) \rightarrow 0$ and
$$
d(\p,\f_j) \leq d(\p,\p_j)+d(\p_j,\f_j) \leq d(\p,\p_j)+2^{1-j} \rightarrow 0.
$$
Recalling that $\p_j \leq \f_j$, it follows from the quasi-triangle inequality that
$$
I(\p,\f_j) \leq c_n \left\{ I(\p,\p_j)+I(\p_j,\f_j) \right\}
\leq c_n \left\{ I(\p,\p_j)+2^{2+n/2} d(\p_j,\f_j)  \right\}\rightarrow 0,
$$
and the proof is complete.
\end{proof}

Recall that the {\it precompletion} of a metric space $(X,d)$
is the set of all Cauchy sequences $C_X$ of $X$, together with the semi-distance
$$
\d(\{x_j\},\{y_j\})=\lim_{j \rightarrow +\infty} d(x_j,y_j).
$$
The metric {\it completion} $(\overline{X},d)$ of $(X,d)$ is the quotient space $C_X/\sim$, where
$$
\{x_j\} \sim \{y_j\} \Longleftrightarrow \d(\{x_j\},\{y_j\})=0,
$$
equipped with the induced distance that we still denote by $d$.

Recall that a {\it path metric space} is a metric space for which the distance between any two points coincides with the infimum of the lengths of rectifiable curves joining the two points. By construction
the space $(\H,d)$ is a path metric space. For such metric spaces, an alternative description of the metric completion can be obtained as follows: consider $C_X'$ the set of all rectifiable curves 
$\g:(0,1] \rightarrow X$ equipped with the semi-distance
$$
\d(\g,\tilde{\g}):=\lim_{t \rightarrow 0} d(\g(t),\tilde{\gamma}(t)).
$$
The metric completion $(\overline{X},d)$ is then the quotient space $C_X'/\sim$ which identifies
zero-distance curves $\g,\tilde{\g}$.

\smallskip

Both constructions yield a rather abstract view on the metric completion. We are now taking advantage of the fact that $\H$ leaves inside the complete metric space $(\E^2(\a),d)$ to conclude that:

\begin{thm}
The metric completion $(\overline{\H}_\a,d)$  is isometric to
$(\E^2(\a),d)$.
\end{thm}


\begin{proof}
We work at the level of normalized potentials, with
$$
 \E^2_{norm}(X,\omega)=\{ \f \in \E^2(X,\omega) \, | \, \sup_X \f=0 \}
$$
and
$$
\H_{norm}=\{ \f \in {\mathcal C}^{\infty}(X,\R) \, | \, \omega+dd^c \f >0 \text{ and } \sup_X \f=0 \}.
$$

Since $(\E^2_{norm}(X,\omega),d)$ is a complete metric space that contains 
$\H_{norm}$, it suffices to show that the latter is dense in $\E^2_{norm}(X,\omega)$.
Fix $\f \in \E^2_{norm}(X,\omega)$ and let $(\f_j) \in {\H}_{norm}^\N$ be a
sequence decreasing to $\f$ (the existence of such a sequence follows from Demailly's regularization
result \cite{Dem92}). It follows from Proposition \ref{prop:upperbound} that
$$
d(\f_{j+p},\f_j)^2 \leq \int_X (\f_{j+p}-\f_j)^2 MA(\f_{j+p}).
$$
Now \cite[Lemma 3.5]{GZ07} shows that the latter is bounded from above by
$$
3^n \int_X (\f-\f_j)^2 MA(\f)
$$
which converges to zero as $j \rightarrow +\infty$, as follows from the monotone convergence theorem.
Therefore $(\f_j)$ is a Cauchy sequence in $({\H}_{norm} ,d)$ which converges to $\f$ since 
$$
0 \leq d(\f,\f_j) \leq 2 I_2(\f_j,\f) \rightarrow 0
$$ 
by Proposition \ref{prop:upperbound} and \cite[Theorem 2.17]{BEGZ10}.

\smallskip

We note the following alternative approach of independent interest.
One  first shows that ${\H}_{norm}$ is dense in the set of all bounded $\omega$-psh functions.
Given $\f \in \E_{norm}^2(X,\omega)$ one then considers its "canonical approximants" 
$$
\f_j=\max(\f,-j) \in PSH_{norm}(X,\omega) \cap L^{\infty}(X)
$$
which  decrease towards $\f \in \E^2(X,\omega)$.
It follows from Proposition \ref{prop:upperbound} that 
\begin{eqnarray*}
\lefteqn{ d(\f_{j+p},\f_j)^2 \leq  \int_X (\f_{j+p}-\f_j)^2 MA(\f_{j+p})}  \\
&&=\int_{(\f \leq -j-p)} p^2 MA(\f_{j+p}) +\int_{(-j-p<\f<-j)} (\f_{j+p}-\f_j)^2 MA(\f)  \\
&&=\int_{(\f \leq -j-p)} p^2 MA(\f)+\int_{(-j-p<\f<-j)} (\f_{j+p}-\f_j)^2 MA(\f)  \\
&& \leq \int_{(\f<-j)} \f^2 MA(\f),
\end{eqnarray*}
where we have used the maximum principle (MP) together with the fact that since $\f \in \E(X,\omega)$,
$$
\int_{(\f \leq -k)} p^2 MA(\f_{k})=\int_X MA(\f_{k})-\int_{(\f>-k)} MA(\f_{k})=
\int_{(\f \leq -k)} MA(\f),
$$
as follows again from the maximum principle.
We infer that $(\f_j)$ is a Cauchy sequence which converges to $\f$.
\end{proof}

\subsection{Bounded geodesics}

Homogeneous complex Monge-Amp\`ere equations have been intensively studied since the mid $70$'s,
after Bedford and Taylor laid down the foundations of pluripotential theory in \cite{BT76,BT82}.
An adaptation of the classical Perron envelope technique yields the following:

\begin{prop} \label{prop:Bo}
Assume $\f_0,\f_1$ are bounded $\omega$-psh functions. Then
$$
\Phi(x,z):=\sup \{ \p(x,z) \, | \, \p \in PSH(X \times A, \omega) \text{ with }
\lim_{t \rightarrow 0,1} \p \leq \f_{0,1} \}.
$$
is the unique bounded $\omega$-psh function on $X \times A$ solution of the Dirichlet problem
$\Phi_{| X \times \partial A}=\f_{0,1}$ with
$$
(\omega+dd^c_{x,z} \Phi)^{n+1}=0 
\text{ in } X \times A.
$$

Moreover $\Phi(x,z)=\Phi(x,t)$ only depends on $|z|$ and 
$|\dot{\Phi}| \leq \| \f_1-\f_0\|_{L^{\infty}(X)}$.
\end{prop}

\begin{proof}
Recall that $A=[1,e] \times S^1$ is an annulus in $\C$ with complex coordinate $z=e^{t+is}$.
The proof follows from a classical balayage technique, together with a barrier argument: observe that
for $A=\| \f_1-\f_0\|_{L^{\infty}(X)}$, the function  
$$
\chi_t(x)=\max\{ \f_0(x)-A \log |z|, \f_1(x)+A (\log|z|-1) \}
$$
belongs to the family and has the right boundary values. Using that $t \mapsto \Phi_t$ is convex
(by subharmonicity in $z$), one easily deduces the bound on $|\dot{\Phi}|$.
We refer the reader to \cite{PSS12,Bern13} for more details.
\end{proof} 

The function $\Phi$ (or rather the path $\Phi_t \subset PSH(X,\omega) \cap L^{\infty}(X)$)
has been called a {\it bounded geodesic} in recent works (see notably \cite{Bern13}).
We use the same terminology here (although it might be confusing at first), as 
it turns out that bounded
(and later on weak) geodesics are geodesics in the metric sense, i.e. constant speed rectifiable paths which minimize length:

\begin{prop} \label{prop:bddmin}
Bounded geodesics are length minimizing. More precisely, if 
$\f_0,\f_1$ are bounded $\omega$-psh functions and $\Phi(x,z)=\Phi_t(x)$ is the bounded
geodesic joining $\f_0$ to $\f_1$, then for all $t,s \in [0,1]$,
$$
d(\Phi_t,\Phi_s)=|t-s| \, d(\Phi_0,\Phi_1).
$$
\end{prop}

\begin{proof} 
Let $\f_0^{(j)}, \f_1^{(j)} \in \H$ be sequences decreasing respectively to $\f_0, \f_1$. 
It follows from the maximum principle and the uniqueness in Proposition 
\ref{prop:Bo} that $\Phi_{t,j}$ decreases to $\Phi_{t}$ as $j$ increases to $+\infty$.
Chen's theorem shows that
$$
d \left( \Phi_{t,j}, \Phi_{s,j} \right)=|t-s| \, d \left( \Phi_{0,j}, \Phi_{1,j} \right).
$$
The conclusion follows since $d \left( \Phi_{t,j}, \Phi_{s,j} \right) \longrightarrow d(\Phi_t,\Phi_s)$.
\end{proof}

\begin{rem}
Assume that $\f_0 \in \H$, while $\f_1 \in PSH(X,\omega) \cap L^{\infty}(X)$. Let
$\f_1^{(j)} \in \H$ be a sequence decreasing to $\f_1$. 
Recall that the approximating geodesics $\Phi_{t,j}$ decrease to $\Phi_{t}$ as $j$ increases to $+\infty$.
We infer from Chen's theorem that
$$
d\left(\f_0,\f_1^{(j)} \right)^2=\int_X \left(\dot{\Phi}_{0,j} \right)^2 MA(\Phi_0)
=\int_X \left(\dot{\Phi}_{t,j} \right)^2 MA(\Phi_t),
$$
for all $t \in [0,1]$.
It follows from the convexity of $t \mapsto \Phi_{t,j}$ that for all $x \in X$,
$$
\dot{\Phi}_{0,j}(x) \leq \frac{ \Phi_{t,j}(x)- \Phi_{0,j}(x)}{t}.
$$
Letting $j \rightarrow +\infty$ and then $t \searrow 0^+$, we infer
$$
\limsup_{j \rightarrow +\infty} \dot{\Phi}_{0,j}(x) \leq \dot{\Phi}_{0}(x).
$$

Observe conversely  that by convexity of $t \mapsto \Phi_t$,
$$
\dot{\Phi}_{0}(x) \leq \frac{ \Phi_{t}(x)- \Phi_{0}(x)}{t}
\leq  \frac{ \Phi_{t,j}(x)- \Phi_{0}(x)}{t},
$$
where the last inequality follows from the maximum principle, which insures that 
$\Phi_{t,j}$ decreases to $\Phi_t$. Since $\Phi_0=\f_0=\Phi_{0,j}$ is fixed, we can let
$t$ decrease to zero and obtain that for all $j,x$,
$
\dot{\Phi}_{0}(x) \leq \dot{\Phi}_{0,j}(x).
$

Therefore $\dot{\Phi}_{0,j}(x) {\longrightarrow} \dot{\Phi}_{0}(x)$  as $j \rightarrow +\infty$. Since
these functions are uniformly bounded, the dominated convergence theorem
insures that 
$$
d(\f_0,\f_1)^2 = \int_X \left(\dot{\Phi}_{0} \right)^2 MA(\Phi_0).
$$

One can however not expect that $d(\f_0,\f_1)^2 = \int_X \left(\dot{\Phi}_{t} \right)^2 MA(\Phi_t)$
for $t>0$ as simple examples show. One can e.g. take $\f_0\equiv 0$ and $\f_1=\max(u,0)$, where
$u$ takes positive values and is a fundamental solution of the Monge-Amp\`ere operator 
(i.e. it has isolated singularities and solves $MA(u)=$Dirac mass at some point):   in this case
$MA(\f_0)$ is concentrated on the contact set $(u=0)$ while $\dot{\Phi}_{1} \equiv 0$ on this set
hence
$$
\int_X \left(\dot{\Phi}_{1} \right)^2 MA(\Phi_1)=0.
$$
We thank T.Darvas for pointing this to us.
\end{rem}

\begin{rem}
 Let $(\f_j)$ be a sequence of bounded $\omega$-psh functions uniformly converging to 
some $\omega$-psh function $\f$. Let $\Phi_t$ denote the geodesic joining 
$\f_j$ to $\f$. 
It follows from Proposition \ref{prop:Bo} that  for all $0 \leq t \leq 1$,
$$
d(\f,\f_j) =\sqrt{ \int_X (\dot{\Phi_t})^2 MA(\Phi_t)} \leq ||\f-\f_j||_{L^{\infty}(X)}.
$$
We will see in Example \ref{exa:explicit} that the convergence in the Mabuchi sense is much weaker than the uniform convergence.
\end{rem}

\subsection{Finite energy geodesics}

We now define weak geodesics joining two finite energy endpoints $\f_0,\f_1 \in \E^1(X,\omega)$.
Fix $j \in \N$ and consider
$\f_0^{(j)}, \f_1^{(j)}$ smooth (or bounded) $\omega$-psh functions decreasing to $\f_0,\f_1$.
We let $\f_{t,j}$ denote the bounded geodesic joining $\f_0^{(j)}$ to $\f_1^{(j)}$.
 It follows from the maximum principle 
that $j \mapsto \f_{t,j}$ is non-increasing. 
We can thus set
$$
\f_t:=\lim_{j \rightarrow +\infty} \f_{t,j} \in \E^1(X,\omega).
$$

It follows again from the maximum principle that $\f_t$ is independent of the choice of the approximants
$\f_0^{(j)}, \f_1^{(j)}$:
if we set as previously $\Phi(x,z):=\f_t(x)$, with $z=t+is$, then $\Phi$ is a maximal
$\omega$-psh function in $X \times A$, as a decreasing limit of maximal $\omega$-psh functions.
It is thus the unique maximal $\omega$-psh function in $X \times \dot{A}$ with boundary values
$\f_0,\f_1$.
We call it the (unique) {\it weak geodesic} joining $\f_0$ to $\f_1$. 

The $\f_t$'s form a family of finite energy functions, since 
$t\mapsto E(\f_{t,j})$ is affine hence
$$
 (1-t) E(\f_0) +t E(\f_1) \leq E(\f_{t,j})
\text{ for all } j \in \N.
$$
  
These weak geodesics are again {\it metric geodesics}:

\begin{thm}
The space $(\overline{\H},d)$ is a CAT(0) space.
More precisely, given $\f_0,\f_1 \in \overline{\H}$, the weak geodesic $\Phi$ joining 
$\f_0$ to $\f_1$  satisfies,
 for all $t,s \in [0,1]$,
$$
d(\Phi_t,\Phi_s)=|t-s| \, d(\Phi_0,\Phi_1).
$$
\end{thm}

\begin{proof}
Complete CAT(0) spaces are also called Hadamard spaces. Recall that a CAT(0) space is a geodesic space which has non positive curvature in the sense of Alexandrov. Hadamard spaces enjoy many interesting properties (uniqueness of geodesics, contractibility, convexity properties,...see \cite{BH99}).

The proof that $(\overline{\H},d)$ is a geodesic space
is very similar to that of Proposition \ref{prop:bddmin} and left to the reader.
Note that $(\overline{\H},d)$ is a complete path metric space, being the completion of the path metric space
$(\H,d)$. The Hopf-Rinow-Cohn-Vossen theorem (see \cite[Proposition 1.3.7]{BH99}) insures that
a complete {\it locally compact} path metric space is automatically a geodesic space.
Here $(\overline{\H},d)$ is not locally compact (it is merely locally weakly compact), but we have 
a natural candidate for the minimizing geodesics.

We finally note that Calabi and Chen proved in \cite[Theorem 1.1]{CC02} that $(\H,d)$
satisfies the CN inequality of Bruhat-Tits \cite{BT72}. It follows therefore from
\cite[Exercise 1.9.1.c and Corollary 3.11]{BH99} that 
 $(\overline{\H},d)$ is a CAT(0) space.
\end{proof}

\section{The toric case} \label{sec:toric}

Recall that a compact K\"ahler {\it toric} manifold $(X,\omega,T)$ is an equivariant compactification
of the torus $T=(\C^*)^n$ equipped with a $T$-invariant K\"ahler metric $\omega$ which writes
$$
\omega=dd^c \p
\text{ in } (\C^*)^n,
$$
with $\p$ $T$-invariant hence  $\p(z)=F \circ L(z)$ where
$$
L:z \in (\C^*)^n \mapsto (\log |z_1|,\cdots, \log|z_n|) \in \R^n
$$
and $F:\R^n \rightarrow \R$ is strictly convex.

The celebrated Atiyah-Guillemin-Sternberg theorem asserts that the moment map 
$\nabla F:\R^n \rightarrow \R^n$ sends $\R^n$ to a bounded convex polytope
$$
P=\{ \ell_i(s) \geq 0, \; 1 \leq i \leq d \} \subset \R^n
$$
where $d \geq n+1$ is the number of $(n-1)$-dimensional faces of $P$,
$$
\ell_i(s)=\langle s,u_i \rangle -\lambda_i,
$$
with $\l_i \in \R$ and $u_i$ is a primitive element of $\Z^n$, normal to the $i^{th}$
$(n-1)$-dimensional face of $P$.

Delzant observed in \cite{Del88} that in this case $P$ is "Delzant", i.e. there are exactly 
$n$ faces of dimension $(n-1)$ meeting at each vertex, and the corresponding $u_j$'s form a 
$\Z$-basis of $\Z^n$. He conversely showed that there is exactly one (up to symplectomorphism)
compact toric K\"ahler manifold $(X_P,\{\omega_P\},T)$ associated to a Delzant polytope $P \subset \R^n$.
Here $\{ \omega_P\}$ denotes the cohomology class of the $T$-invariant K\"ahler form  $\omega_P$.
Let 
$$
G(s):=\sup_{x \in \R^n} \{ \langle x,s \rangle-F(x) \}
$$
denote the Legendre transform of $F$. Observe that  $G=+\infty$ in $\R^n \setminus P$ and for
$s \in P=\nabla F(\R^n)$,
$$
G(s)= \langle x,s \rangle-F(x) 
\; \text{ with } \; 
\nabla F(x)=s \Leftrightarrow \nabla G(s)=x.
$$
Guillemin observed in \cite{Gui94} that a "natural" representative of the cohomology class $\{\omega_P\}$
is given by
$$
G(s)=\frac{1}{2} \left\{ \sum_{i=1}^d \ell_i(s) \log \ell_i(s) +\ell_{\infty}(s)\log \ell_{\infty}(s) \right\}
$$
where $\ell_{\infty}(s)=\sum_{i=1}^d \langle s,u_i \rangle$. 
We refer the reader to \cite{CDG03} for a neat proof of this beautiful Guillemin formula.

\begin{exa} \label{exa:proj}
When $X=\C\P^n$ and $\omega$ is the Fubini-Study K\"ahler form, then
$$
F(x)=\frac{1}{2} \log \left[ 1+\sum_{i=1}^n e^{2x_i} \right],
$$
$P=\nabla F(\R^n)$ is the simplex
$$
P=\left\{s_i \geq 0, \, 1 \leq i \leq n \text{ and } \sum_{i=1}^n s_i \leq 1\right\},
$$
thus $d=n+1$, 
$$
\ell_i(s)=s_i, \; \l_i=0, \; u_i=e_i 
\text{ for } 1 \leq i \leq n
$$
and
$$
\ell_{n+1}(s)=1-\sum_{i=1}^n s_i, \; \l_{n+1}=-1, \, e_{n+1}=-\sum_{j=1}^n e_j
$$
and $\ell_{\infty} \equiv 0$ so that
$$
G(s)=\frac{1}{2} \left\{ \sum_{i=1}^n s_i \log s_i+\left(1-\sum_{j=1}^n s_j \right)
\log \left( 1-\sum_{j=1}^n s_j \right) \right\}.
$$
\end{exa}

\subsection{Toric geodesics}

Let $(X,\omega,T)$ be a compact toric manifold. If $\f_0,\f_1 \in \H$ are both 
$T$-invariant, it follows from the uniqueness that the geodesic $(\f_t)_{0 \leq t \leq 1}$
consists of $T$-invariant functions. Let $F_t$ denote the corresponding potentials
in $\R^n$ so that
$$
F_t \circ L=F \circ L+\f_t 
\text{ in } (\C^*)^n.
$$

\begin{prop} \cite{Guan99}
The map $(x,t) \mapsto \f_t(x)$ is smooth and corresponds to the Legendre transform of an affine path
on $P$.
\end{prop}

In other words the Legendre transform $G_t$ of $F_t$ is affine in $t$. The proof is elementary, it suffices to differentiate twice the defining equation for $G_t$. In a similar vein we obtain an explicit formula for the Mabuchi distance between $\f_0$ and $\f_1$:

\begin{prop}  \label{pro:distancetorique}
$$
d(\f_0,\f_1)=||G_1-G_0||_{L^2(P)}=\sqrt{\int_P (G_1-G_0)^2(s) ds}.
$$
\end{prop}

\begin{proof}  
Recall that
$$
d(\f_0,\f_1)=\sqrt{\int_X (\dot{\f_0})^2 MA(\f_0)}.
$$
Now $F_t \circ L=F \circ L+\f_t$ has Legendre transform $G_t=tG_1+(1-t)G_0$. Thus
$\dot{\f_t}=\dot{F_t} \circ L$ with
$$
G_t(s)=\langle x_t,s \rangle -F_t(x_t)
\text{ with } s=\nabla F_t(x_t)
$$
hence 
$
\dot{G_t}(s)=-\dot{F_t}(x)
$
and we infer
$$
d(\f_0,\f_1)^2=\int_{(\C^*)^n} (\dot F_0)^2 MA(F_0 \circ L).
$$
Observe that
$$
\frac{\partial^2 (F_0 \circ L)}{\partial z_i \partial \overline{z_j}}=
\frac{1}{\overline{z_i} z_j} \cdot \frac{\partial^2 F_0}{\partial z_i \partial \overline{z_j}} \circ L 
\; \text{ in } \; (\C^*)^n
$$
hence
$$
\det \left( \frac{\partial^2 (F_0 \circ L)}{\partial z_i \partial \overline{z_j}} \right)
=\frac{1}{\Pi_j |z_j|^2} \cdot MA_{\R}(F_0) \circ L,
$$
where $MA_{\R}$ denotes the real Monge-Amp\`ere measure (in the Alexandrov sense, see
\cite{Gut01}) of the convex function $F_0$. Thus
$$
\int_{(\C^*)^n} (\dot F_0)^2 MA(F_0 \circ L)=\int_{\R^n} (\dot F_0)^2 MA_{\R}(F_0) .
$$

Now $\dot{F_0}=-\dot{G_0} \circ \nabla F_0$ and $MA_{\R}(F_0)=(\nabla F_0)^* ds$ therefore
$$
\int_{\R^n} (\dot F_0)^2 MA_{\R}(F_0)=\int_P (\dot{G_0})^2(s) ds=\int_P (G_1-G_0)^2(s)ds.
$$
\end{proof}

\begin{exa} \label{exa:explicit}
Assume $X=\C\P^1$ is the Riemann sphere and $\omega$ is the Fubini-Study K\"ahler form.
Let $\f_0$ be the toric function associated to the convex potential
$$
F_0(x)=\max(x,0)
\text{ so that } 
G_0(s) \equiv 0
\text{ on the simplex } P=[0,1].
$$
Observe that $\omega_0=dd^c F_0 \circ L$ is the (normalized) Lebesgue measure on the unit circle
$S^1 \subset \C^* \subset \C\P^1$.

We consider $\f_1=\f_j$ a sequence of toric potentials defined by the convex functions
$$
F_j(x)=(1-\e_j) F_0(x)+\e_j \max(x,-C_j),
$$
where $\e_j$ decreases to $0$, while $C_j$ increases to $+\infty$. A straightforward computation yields
$$
G_j(s)=\max( C_j[\e_j-s], 0).
$$
Therefore
$$
d(\f_j,\f_0)=\frac{C_j \e_j^{3/2}}{\sqrt{3}}
$$
We thus obtain in this case, as $j \rightarrow +\infty$,
\begin{itemize}
\item $\f_j \longrightarrow \f_0$ in $L^1$ iff $\e_j \rightarrow 0$;
\item $\f_j \longrightarrow \f_0$ in $L^{\infty}$ iff $\e_j C_j\rightarrow 0$;
\item $\f_j \longrightarrow \f_0$ in $(\E^1,I_1)$ iff $\e_j^2C_j \rightarrow 0$;
\item $\f_j \longrightarrow \f_0$ in $(\E^2,I_2)$ iff $\e_j^3C_j^2 \rightarrow 0$;
\item $\f_j \longrightarrow \f_0$ in $(\H,d)$ iff $\e_j^3C_j^2 \rightarrow 0$.
\end{itemize}

The convergence in $(\E^1,I_1)$ is here (i.e. in dimension $n=1$)
the convergence in the Sobolev norm $W^{1,2}$. For $\e_j=1/j$ and $C_j=j^{3/2}$ we 
therefore obtain an example of a sequence which converges in the Sobolev sense
but not in the Mabuchi metric.
\end{exa}

\subsection{Toric singularities}

Let $\f \in \H$ be a toric potential. We are going to read off the singular behavior of $\f$ from the
integrability properties of the Legendre transform of its associated convex potential.

We let $F_\f$ and $G_\f$ denote the corresponding convex function and its Legendre transform. The function $\f$ is bounded (resp. continuous) if and only if so is
$F_\f-F$ on $\R^n$, since $F_\f \circ L=F \circ L+\f$, if and only if so is $G_\f$  on $P$, as
$G$ (Guillemin's potential) is continuous on $P$.

The same conclusion holds if we take as a reference potential the support function $F_P$ of $P$, defined by
$$
F_P(x):=\sup_{s \in P} \langle s, x \rangle.
$$
It is the Legendre transform of the function $G_P$ which is identically $0$ on $P$ and $+\infty$ in $\R^n \setminus P$. We can similarly understand finite energy classes:

\begin{prop} \label{pro:energytorique}
$$
\f \in PSH_{tor}(X,\omega) \cap L^{\infty}(X) \Longleftrightarrow G_\f \in L^{\infty}(P).
$$
$$
\f \in \E^q_{tor}(X,\omega) \Longleftrightarrow G_\f \in L^q(P).
$$
\end{prop}

We refer the reader to \cite[Proposition 2.9]{BerBer13} for an elegant proof of this result when $q=1$.

\begin{proof}
We first show that $\f \in \E^q_{tor}(X,\omega) \Longrightarrow G_\f \in L^q(P)$. Approximating
$\f$ from above by a decreasing sequence of smooth strictly $\omega$-psh toric functions,
this boils down to show a uniform a priori bound
$$
||G_\f||_{L^q(P)} \leq \left( \int_X |\f_P-\f|^q MA(\f) \right)^{1/q}.
$$

We can assume without loss of generality that $F \leq F_P$ (since $\f$ is upper semi-continuous hence
bounded from above on $X$ which is compact).
Recall that $\f=(F_\f-F_P) \circ L$ in $(\C^*)^n$, where $F_P$ denotes a reference potential associated
to $\omega$. Changing variables and using the Legendre transform yields
\begin{eqnarray*}
\int_{(\C^*)^n} |\f-\f_P|^q MA(\f)&=&\int_{\R^n} |F-F_P|^q MA_{\R}(F) \\
&=&\int_P |F \circ \nabla G (s)-F_P \circ \nabla G(s)|^q ds,
\end{eqnarray*}
where
$F(x)=\langle x, s \rangle -G(s)$, with $\nabla G(s)=x$. Therefore
$$
F(\nabla G(s))=\langle \nabla G(s), s \rangle -G(s)
$$
and
\begin{eqnarray*}
F_P(\nabla G(s))-F(\nabla G(s))&=&G(s)-\{ \langle \nabla G(s), s \rangle-F_P \circ \nabla G(s) \} \\
&\geq& G(s)-G_P(s)=G(s) \geq 0,
\end{eqnarray*}
since $G_P(s)=\sup_{x \in \R^n} \{ \langle x, s \rangle-F_P (x) \} =0$ for $s \in P$. We infer
\begin{eqnarray*}
||G_\f||^q_{L^q(P)} &\leq& \int_P \left| F_P(\nabla G(s))-F(\nabla G(s)) \right|^q ds \\
& \leq & \int_X |\f_P-\f|^q MA(\f),
\end{eqnarray*}
as claimed.

We now take care of the converse implication. To simplify notations we only treat the case of the complex projective space, using the same notations as in Example \ref{exa:proj}. We want to get an
upper bound on $\int_X |\f_0-\f|^q MA(\f)$ involving $\int_P |G-G_0|^q(s)ds$, where
$$
F_0(x)=\frac{1}{2} \log \left[ 1+\sum_{i=1}^n e^{2x_i} \right]
$$
is the convex function associated to the Fubini-Study K\"ahler form,
$$
P=\left\{s_i \geq 0, \, 1 \leq i \leq n \text{ and } \sum_{i=1}^n s_i \leq 1\right\}
\text{ is the standard simplex}
$$
and
$$
G_0(s)=\frac{1}{2} \left\{ \sum_{i=1}^n s_i \log s_i+\left(1-\sum_{j=1}^n s_j \right)
\log \left( 1-\sum_{j=1}^n s_j \right) \right\}.
$$

We decompose $(\C^*)^n$ in $2^n$ pieces, according to whether $|z_j| \leq 1$ or  $\geq 1$.
By symmetry, it suffices to bound $\int_{\Delta^n} |\f_0-\f|^q MA(\f)$, where
$$
\Delta^n=\{ z \in (\C^*)^n, \, |z_j| \leq 1 \text{ for all } 1 \leq j \leq n \}
$$
is sent by $L$ to $\R_-^n$. Now $F_0$ is bounded on $\R_-^n$, 
$0 \leq F_0 \leq \log (n+1) /2$, so 
$$
\int_{\R_-^n} |F_0-F|^q MA_{\R}(F) 
$$
is bounded from above if and only if so is $\int_{\R_-^n} |F|^q MA_{\R}(F) $. Applying the Legendre transform yields
$$
\int_{\R_-^n} |F|^q MA_{\R}(F)=\int_{\nabla F(\R_-^n)} |G|^q(s) ds \leq ||G||^q_{L^q(P)},
$$
as desired.
\end{proof}

\begin{thm} \label{thm:toric}
The metric completion of $(\H_{tor},d)$ is $(\E^2_{tor},d)$.
\end{thm}

\begin{proof}
We have already observed that the metric completion of $(\H_{tor},d)$ can be identified to
a subset of $\E^1_{tor}$. It follows from 
Proposition \ref{pro:distancetorique} 
and Proposition \ref{pro:energytorique} that the only functions lying at finite distance from smooth toric potentials are those which belong to $\E^2_{tor}$.
\end{proof}

\subsection{Concluding remarks}

\subsubsection{Controlling the Sup}

One would like to bound from above the supremum of a function in $\H$ in terms of the Mabuchi distance.
It is natural to expect that there exists a constant $C>0$ such that for all $\f \in \H$,
$$
\sup_X \f  \leq C   \max(1, d(0,\f)  ).
$$
We explain how to establish such a bound in the toric setting:

\begin{lem}
There exists a constant $C_P>0$ such that for all convex functions $G$ defined on the
convex polytope $P$, 
$$
-\inf_P G \leq C_P \cdot ||G||_{L^1(P)}.
$$
\end{lem}

Cauchy-Schwarz inequality then yields
$$
||G_\f||_{L^1(P)} \leq \vol(P) \cdot ||G_\f||_{L^2(P)}=\vol(P) \cdot d(0,\f).
$$
Note finally that 
$
\sup_X \f =-\inf_P G_\f+O(1)
$
to conclude.

\begin{proof}
It is likely that such an inequality is well known and classical.
We explain the argument in the simple case of the interval $P=[0,1]$.

By approximation it suffices to consider the case of convex functions that are continuous up to the
 boundary.   We can assume that the maximum of $G$ is non positive:
observe that the situation becomes worse if we translate the graph of $G$ to the south.

Let $M$ be a point realizing the minimum of $G$ and draw the triangle $T$ joining 
the points $(0,0), M$ and $(1,0)$. By convexity of $G$ this triangle is contained in the 
epigraph of $G$. We thus have
$$
\frac{-\inf G}{2}=\rm{Area}(T) \leq \| G\|_{L^1}.
$$
\end{proof}

\subsubsection{Legendre transform of a minimum}

We let the reader check that the Legendre transform $G_{\f \vee \p}$
of the minimum of two convex functions is the maximum of the Legendre transforms
of these functions,
$$
G_{\f \vee \p}=\max( G_\f,G_\p).
$$

The orthogonality relation observed in Proposition \ref{prop:pyth} translates 
 in the toric setting as follows,
$$
d(\f,\f \vee \p)^2=\int_P \left(G_\f-G_{\f \vee \p} \right)^2
=\int_{\{G_\f <G_\p\}} \left(G_\f-G_{\p} \right)^2,
$$
while
$$
d(\p,\f \vee \p)^2=\
=\int_{\{G_\f >G_\p\}} \left(G_\f-G_{\p} \right)^2,
$$
so that
$$
d(\f,\f \vee \p)^2+d(\f \vee \p,\p)^2=d(\f, \p)^2.
$$

\subsubsection{Comparing $d$ and $I_2$}

It is tempting to think that the Mabuchi distance is equivalent to the strong topology of $\E^2$, but we could only partially prove this. Proposition \ref{prop:upperbound} shows that
$$
d \leq 2 I_2
$$
and conversely there exists a uniform constant $c_n=2^{2+n/2}>0$ such that
$$
I_2(\f,\p) \leq c_n  d(\f,\p)
$$
when $\f \leq \p$. One would like to get rid of this restriction, enlarging $c_n$ if necessary.
This is equivalent to checking that $I_2$ satisfies a quasi-triangle inequality.

Using Proposition \ref{prop:pyth} we can reformulate the inequalities obtained so far
by introducing the quantity $J_2(\f,\p)$ defined by
$$
J_2(\f,\p):=\sqrt{ I_2(\f,\f \vee \p)^2+I_2(\f \vee \p, \p)^2}.
$$
We let the reader check that  $J_2 \leq I_2$ and
$$
d \leq 2 J_2 \leq 2^{2+n/2} d.
$$

\end{document}